\documentclass[brochure,english,12pt]{bourbaki}

\usepackage[T1]{fontenc}
\usepackage{lmodern,amssymb,bm}
\usepackage{comment}
\usepackage{amsthm}
\usepackage{amsmath}
\usepackage{amsfonts}
\usepackage{amssymb}
\usepackage{mathtools}
\usepackage{graphicx}
\usepackage{chngcntr}
\usepackage{url}
\usepackage{hyperref}
\usepackage[makeroom]{cancel}
\usepackage{enumerate}
\usepackage{mathrsfs}
\usepackage{tikz}
\usepackage{tikz-cd}
\usetikzlibrary{decorations.markings}
\usetikzlibrary{decorations.pathreplacing}
\usetikzlibrary{patterns}
\usepackage[utf8]{inputenc}

\usepackage[
backend=biber,
style=authoryear, 
citestyle=authoryear-comp
]{biblatex}

\usepackage{csquotes}

\DeclareDelimFormat{nameyeardelim}{\addcomma\space}

\DeclareNameAlias{sortname}{given-family}

\AtEveryBibitem{%
  \clearfield{issn} 
  \clearfield{doi} 

  \ifentrytype{online}{}{
    \clearfield{url}
  }
}

\renewbibmacro{in:}{%
    \ifentrytype{article}{}{\printtext{\bibstring{in}\intitlepunct}}}

\DeclareFieldFormat[article,periodical,inreference]{number}{\mkbibparens{#1}}
\DeclareFieldFormat[article,periodical,inreference]{volume}{\mkbibbold{#1}}
\renewbibmacro*{volume+number+eid}{%
    \printfield{volume}%
    \setunit*{\addthinspace}
    \printfield{number}%
    \setunit{\addcomma\space}%
    \printfield{eid}}

\DeclareFieldFormat[article,inbook,incollection]{title}{\enquote{#1}\addcomma} 

\date{Mars 2020}
\bbkannee{72\textsuperscript{e} ann\'ee, 2019--2020}  
\bbknumero{1175}                                      

\title{Quantitative inverse theory of Gowers uniformity norms}
\subtitle{after F. Manners}

\author{Thomas F. Bloom}
\address{Centre for Mathematical Sciences\\ Cambridge University\\ UK}
\email{tb634@cam.ac.uk}

\theoremstyle{plain}
\newtheorem{definition}{\definame}[section]
\newtheorem{proposition}[defi]{\propname}
\newtheorem{theorem}[defi]{\theoname}

\newtheorem{question}[defi]{Question}

\newtheorem{lemma}[defi]{\lemmname}
\theoremstyle{remark}

\newcommand{\abs}[1]{\left\lvert #1\right\rvert}
\newcommand{\brac}[1]{\left( #1\right)}

\newcommand{\norm}[1]{\left\lVert #1\right\rVert}

\newcommand{\bbf}{\mathbb{F}}

\newcommand{\bbz}{\mathbb{Z}}
\newcommand{\bbr}{\mathbb{R}}

\newcommand{\bbc}{\mathbb{C}}

\newcommand{\td}{\,\mathrm{d}}
\newcommand*{\bbe}{
  \mathop{
    \mathchoice{\vcenter{\hbox{\larger[4]$\mathbb{E}$}}}
               {\kern0pt\mathbb{E}}
               {\kern0pt\mathbb{E}}
               {\kern0pt\mathbb{E}}
  }\displaylimits
}

\newcommand{\bvec}[1]{\mathbf{#1}}

\begin{document}

\maketitle

\section*{Abstract:}
Gowers uniformity norms are the central objects of higher order Fourier analysis, one of the cornerstones of additive combinatorics, and play an important role in both Gowers' proof of Szemer\'{e}di's theorem and the Green-Tao theorem. The inverse theorem states that if a function has a large uniformity norm, which is a robust combinatorial measure of structure, then it must correlate with a nilsequence, which is a highly structured algebraic object. This was proved in a qualitative sense by Green, Tao, and Ziegler, but with a proof that was incapable of providing reasonable bounds. In 2018 Manners achieved a breakthrough by giving a new proof of the inverse theorem. Not only does this new proof contain a wealth of new insights but it also, for the first time, provides quantitative bounds, that are at worst only doubly exponential. This talk will give a high-level overview of what the inverse theorem says, why it is important, and the new proof of Manners.

\section*{Introduction}
Let $s\geq1$ be a fixed integer (in particular, all implicit constants may depend on $s$). We fix some large prime $N$ and use $\bbz_N$ to denote the cyclic group $\bbz/N\bbz$. If $f:\bbz_N\to\bbc$ and $h\in \bbz_N$ we define the multiplicative derivative $\Delta_hf:\bbz_N\to \bbc$ by $\Delta_h f(x)=f(x)\overline{f(x+h)}$. This definition has a natural extension to vectors $\bvec{h}\in \bbz_N^k$, so that
\[\Delta_{\bvec{h}}f(x) = \Delta_{h_1}\cdots \Delta_{h_k}f.\]
We can now define our central object of study: the Gowers uniformity norm of degree $s$. This is defined for $f:\bbz_N\to\bbc$ by

\[\norm{f}_{U^{s+1}}=\brac{ \bbe_{\bvec{h}\in \bbz_N^{s+1}}\bbe_{x\in \bbz_N}\Delta_{\mathbf{h}}f(x)}^{1/2^{s+1}}.\]

Here we have used the expectation notation, now standard within additive combinatorics, to denote the normalised sum (so that, for example, $\bbe_{x\in H}$ means $\frac{1}{\abs{H}}\sum_{x\in H}$). These norms\footnote{It is not obvious from the definition that they are indeed norms, but this is true for $s\geq 1$. A proof can be found in the original paper of \textcite{Go01}.} were introduced into additive combinatorics by \textcite{Go01} in his analytic proof of Szemer\'{e}di's theorem. They provide a quantitative measure of the additive structure of $f$. For example, the $U^2$ norm measures $f$ along $2$-dimensional additive quadruples of the shape $(x,x+h_1,x+h_2,x+h_1+h_2)$. In general, the $U^{s+1}$ norm measures $f$ along $2^{s+1}$-tuples which are (the projections of) $(s+1)$-dimensional cubes.

These configurations are generic enough that if we can understand how $f$ behaves on them then we can understand $f$ on almost any other kind of linear system. For example, the count of $f$ along arithmetic progressions of length $k$ is controlled by $\norm{f}_{U^{k-1}}$. As a result the uniformity norms play a central role in modern additive combinatorics. They often appear in proofs which use a `structure vs.\ randomness' philosophy. Such proofs go along the following lines:
\begin{enumerate}
\item Find some function $f$ such that behaviour of the objects one wishes to understand is governed by some $\norm{f}_{U^{s+1}}$ (for example, when counting $k$-term arithmetic progressions inside some $A\subset \bbz_N$ one would take $f=1_A-\tfrac{\abs{A}}{N}$ and $s=k-2$);
\item Show that the theorem in question follows if $\norm{f}_{U^{s+1}}$ is \emph{small} -- this is the random aspect of the `structure vs.\ randomness' dichotomy, and usually follows from elementary counting methods;
\item Show that if $\norm{f}_{U^{s+1}}$ is \emph{large} then $f$ has some algebraic structure;
\item Finally conclude the proof by showing that if $f$ has this algebraic structure then the theorem also follows.
\end{enumerate}

Landmark applications of this method are the analytic proof of Szemer\'{e}di's theorem by \textcite{Go01}, the first to deliver reasonable quantitative bounds, and the proof by \textcite{GT08a} that the primes contain arbitrarily long arithmetic progressions. 

The inverse theory of uniformity norms, which is the focus of this article, addresses the third point, and seeks to answer the following question.

\begin{question}[The Inverse Question]
If a $1$-bounded\footnote{A function $f$ is $1$-bounded if $\abs{f(x)}\leq 1$ for all $x$.} function $f:\bbz_N\to\bbc$ has large $U^{s+1}$ norm then what can we deduce about $f$?
\end{question}

We will first explore this question in the simplest case $s=1$. A brief experimentation shows that the function\footnote{We use $e(x)$ to denote $e^{2\pi ix}$.} $x\mapsto e(\alpha x)$, where $\alpha \in\tfrac{1}{N}\bbz$ (so that this does indeed give a well-defined function on $\bbz_N$), satisfies $\norm{f}_{U^2}=1$, the maximum possible. Therefore (linear) characters provide examples of bounded functions with large $U^2$ norm. The inverse theorem for the $U^2$ norm says that the converse is also true, at least in an approximate sense. More precisely, if $\norm{f}_{U^2}\geq \delta$ then $f$ must correlate with a linear character, in that there exists some $\alpha\in \tfrac{1}{N}\bbz$ such that 
\begin{equation}\label{u2inv}
\abs{\bbe_x f(x)\overline{e(\alpha x)}}> c(\delta)
\end{equation}
for some constant $c(\delta)>0$ depending only on $\delta$. The proof of this inverse result is a one-line consequence of orthogonality. We note first that the $U^2$ system of quadruples of the shape $(x,x+h_1,x+h_2,x+h_1+h_2)$ are exactly those quadruples $(x_1,x_2,x_3,x_4)$ such that $x_1+x_2=x_3+x_4$. We may use linear characters and orthogonality to detect this latter linear condition to see that, defining the Fourier transform by $\widehat{f}(a)=\bbe_x f(x)e(-ax/N)$ for $a\in\bbz_N$, we have
\begin{align*}
\norm{f}_{U^2}^4
&= \sum_{a\in \bbz_N}\abs{\widehat{f}(a)}^4\\
&\leq \sup_a \abs{\widehat{f}(a)}^2 \sum_{a\in \bbz_N}\abs{\widehat{f}(a)}^2\\
&= \sup_a \abs{\widehat{f}(a)}^2\bbe_{x\in \bbz_N}\abs{f(x)}^2\\
&\leq \sup_a\abs{\widehat{f}(a)}^2,
\end{align*}
using Parseval's identity and the assumption that $f$ is $1$-bounded. It follows that \eqref{u2inv} holds with the explicit lower bound $\geq \delta^2$.

For higher uniformity norms with $s\geq 2$ this simple argument fails, and it is less clear what functions have large $U^{s+1}$ norm. We note that the uniformity norms are nested, in the sense that
\[\norm{f}_{U^2}\leq \norm{f}_{U^3}\leq \norm{f}_{U^4}\leq \cdots,\]
and so in particular this inverse question becomes more difficult as $s$ increases, since any example of a function with large $U^{s+1}$ norm certainly also has large $U^{s+2}$ norm, but the converse may not hold.

Considering what happens when $s=1$, we observe that the reason exponentials with linear phases have large $U^2$ norm is because the second derivative of a linear function always vanishes (hence $\Delta_{h_1,h_2}f$ is an exponential with phase $0$, and so is identically 1). Generalising this observation shows that $f(x)=e(P(x))$ will have $\norm{f}_{U^{s+1}}=1$ whenever $P$ is a polynomial of degree $\leq s$ (with coefficients in $\tfrac{1}{N}\bbz$ so that it is well-defined on $\bbz_N$).

At this point one may guess that, just as when $s=1$, the approximate converse also holds, and conjecture something like: if $\norm{f}_{U^{s+1}}\geq \delta$ then 
\begin{equation}\label{usinv}
\abs{\bbe_x f(x)\overline{e(P(x))}}> c(\delta)
\end{equation}
for some polynomial $P(x)\in \tfrac{1}{N}\bbz[x]$ of degree $\leq s$ and some constant $c(\delta)>0$ depending only on $\delta$. It turns out this is not quite enough, and the exponentials with polynomial phases do not represent a full set of obstructions.\footnote{The term obstructions here comes from thinking of functions $g$ such that ``$\lvert \bbe_x f(x)\overline{g(x)}\rvert\gg 1$ implies $\norm{f}_{U^{s+1}}\gg 1$'' as `obstructions' to having small uniformity norm.} To see why, observe that we don't require our `phase functions' $P(x)$ to actually vanish after taking $s+1$ derivatives. For the $U^{s+1}$ norm of $e(P(x))$ to be large, it suffices for the derivative of $P(x)$ to be biased towards $0$ (so that the multiplicative derivative of $e(P(x))$ is biased towards 1). Consider, for example, the function on $\bbz_N$ defined by\footnote{As usual in number theory, we use $\{ \cdot\}: \bbr \to [0,1)$ to denote the fractional part operator $x\mapsto x-\lfloor x\rfloor$, where $\lfloor x\rfloor$ is the largest integer $n$ such that $n\leq x$.}
\[x \mapsto \{\alpha x\}\{ \beta x\}\]
for $\alpha,\beta\in \tfrac{1}{N}\bbz$. This is not a true quadratic, since the third derivative is not identically zero. This failure arises because $x\mapsto\{\alpha x\}$ is not truly linear. We note, however, that it is linear a large proportion of the time, since we can write $\{\alpha (x+y)\}=\{\alpha x\}+\{\alpha y\}-\rho_{x,y}$, where $\rho_{x,y}$ is a `carry bit' function that is $1$ if $\{\alpha x\}+\{\alpha y\}\geq 1$ and $0$ otherwise. In particular, with probability $1/2$, $\rho_{x,y}=0$ and hence $\{\alpha x\}$ behaves like a linear function with probability $1/2$. By a similar calculation, the third derivative of a quadratic function like $P(x)=\{\alpha x\}\{\beta x\}$ vanishes a positive proportion of the time, and thus $e(P(x))$ has large $U^3$ norm. Thus we need to also include these generalised bracket polynomials (generated by repeated composition of, not just addition and multiplication, but also the fractional part operator $\{\cdot\}$) as possible obstructions.

The inverse theorem for the Gowers uniformity norms states that this expanded set of obstructions captures all the reasons that a function might have large uniformity norm. That is, if $\norm{f}_{U^{s+1}}\geq \delta$ then there exists some bracket polynomial $P$ of degree $\leq s$ such that \eqref{usinv} holds. Such a statement was conjectured by \textcite{GT10} in their work on linear equations in primes and first proved in a qualitative sense (that is, with no bounds on the function $c(\delta)$) by \textcite{GTZ12}. 

The focus of this article is a recent new proof of the inverse theorem by \textcite{Ma18} which gives, for the first time, a quantitative version of this statement. Our aims are threefold:
\begin{enumerate}
\item To state precisely the inverse theorem as proved in \textcite{Ma18}, defining all the concepts required;
\item To sketch some of the ideas used in the proof, giving a flavour of the kind of arguments used; and
\item To state precisely some of the new definitions and concepts introduced by Manners.
\end{enumerate}
The paper of \textcite{Ma18} itself is over 100 pages long, so we will not be able to even approach a proper proof of the inverse theorem. As a result of this intimidating length, however, some of the beautiful new ideas introduced by Manners may go otherwise unnoticed by those without the time to plumb the technical depths. We hope that this article helps popularise them, so that they can find many other applications. 

\section*{Acknowledgements}

I would like to thank Freddie Manners for many helpful comments and guidance on understanding his work, and Aled Walker for his notes and insights on Manners' proof, which have helped me immensely. I would also like to thank an anonymous referee, Gabriel Conant, Ben Green, and Sarah Peluse for their comments on an earlier draft.

\section{The inverse theorem}
The following quantitative inverse theorem for the Gowers uniformity norms is the main result of \textcite{Ma18}, and the focus of this article.

\begin{theorem}[Manners]\label{thmain}Let $\delta>0$ and $N\geq 2$ be prime. If $f:\bbz_N\to\bbc$ is a $1$-bounded function such that $\norm{f}_{U^{s+1}}\geq \delta$ then there exists $\epsilon>0$ and a $1$-bounded, $N$-periodic, nilsequence $\psi:\bbz\to\bbc$ with degree $s$, dimension $D$, parameter $K$, and complexity $M$, such that
\[\abs{\bbe_x f(x)\overline{\psi(x)}}\geq \epsilon,\]
where the parameters are bounded in terms of $\delta$ by
\[D =O(\delta^{-O(1)})\textrm{ and } 
\epsilon^{-1},K,M\leq \begin{cases} \exp(O(\delta^{-O(1)}))&\textrm{ if }s\leq 3\textrm{ and }\\ \exp(\exp(O(\delta^{-O(1)})))&\textrm{ if }s\geq 4.\end{cases}\]
\end{theorem}

We have followed Manners in stating the inverse result using nilsequences rather than bracket polynomials. These are qualitatively equivalent, as shown by \textcite{BL07}.\footnote{\textcite{BL07} do not give any quantitative form of this equivalence, which would be needed to make a version of Theorem~\ref{thmain} precise for bracket polynomials. Since nilsequences are more convenient to work with anyway we will not address this, and use bracket polynomials only as motivational examples.} Nilsequences are usually, however, easier to work with. We defer an explicit definition of nilsequences, along with the meaning of degree, dimension, parameter, and complexity, till Section~\ref{nilob}. For now the reader should think of a nilsequence of degree $s$ as a function of the shape $n\mapsto e(p(n))$, where $p(n)$ is a (bracket) polynomial of degree at most $s$.

Since it can also be shown that if $f$ correlates with a nilsequence of degree $s$ then it has large $U^{s+1}$ norm (we sketch a proof of this in Section~\ref{secob}), this gives a complete characterisation of functions with large uniformity norm. A characterisation of this type was first conjectured by \textcite[Conjecture~8.3]{GT10} in their work on linear equations in primes. 

Before we explore Theorem~\ref{thmain} and its proof we briefly summarise other, related, results. 
\begin{enumerate}
\item The case $s=1$ is a trivial argument using Fourier analysis, which we have already seen in the introduction. Observe that this argument gives bounds which are polynomial in $\delta$.
\item The case $s=2$ was first established by \textcite{GT08b}, with bounds that are quasipolynomial in $\delta$. The inverse theory for the $U^3$ norm has a strong connection with Freiman-type inverse sumset problems. For example, it is known that polynomial bounds for the inverse $U^3$ problem are equivalent to the Polynomial Freiman-Ruzsa conjecture (as shown by \textcite{GT10b} and \textcite{Lo12}). 
\item The case for general $s$ was proved in a qualitative sense by \textcite{GTZ12}. This proof gives no explicit bounds at all, and uses the language of non-standard analysis. In theory this proof could be made explicit, but any bounds that could be extracted would be terrible, and far up in the Ackermann hierarchy (in particular, much worse than even a tower of exponentials of height $O(\delta^{-1})$).
\item An independent approach towards a proof of the inverse theorem for general $s$ was initiated by \textcite{Sz12} and further developed in a number of papers by \textcite{CS10,CS19,GMVa,GMVb,GMVc}. This approach is also qualitative, and would deliver similarly terrible bounds. 
\item These inverse results are all `global', in that they find a structured object with which $f$ correlates on the entirety of $\bbz_N$. A weaker local inverse statement, in which one just finds some large subset of $\bbz_N$ on which $f$ has strong correlation with a structured object, was proved by \textcite[Theorem 18.1]{Go01} (with even better bounds, only polynomial in $\delta$). This local statement was sufficient to prove Szemer\'{e}di's theorem, but for most applications of the uniformity norms (such as the \textcite{GT10} work on linear equations in primes) a global inverse statement is required.
\item Similar uniformity norms and their inverse results have been studied in ergodic theory, where they were introduced independently of additive combinatorics and the work of Gowers. The ergodic theory analogue of the inverse result was established by \textcite{BHK05}. This work was the first to recognise the importance of nilsequences as the proper set of obstructions, and shaped the formulation of the inverse conjecture by \textcite{GT10}. Although we will not use ergodic-theoretic language in this article, insights and arguments from ergodic theory have been extremely influential in much of the work on uniformity norms. For an in-depth discussion of the relationship between uniformity norms in additive combinatorics and ergodic theory we refer to \textcite{HK09}. 
\item The discussion so far has been entirely for functions $f:\bbz_N\to \bbc$, but one could define uniformity norms and ask the inverse question with $\bbz_N$ replaced by any finite abelian group. The principal alternative is $\bbf_p^n$, where the inverse theorem was first proved qualitatively by \textcite{TZ10}. Very recently a new proof with good quantitative bounds was given by \textcite{GM17,GM20}. While there are some similarities between the methods used for $\bbz_N$ and $\bbf_p^n$, there are important differences that make it difficult to translate progress in one setting to the other. In particular, the strategy of Manners uses crucially that $\bbz_N$ has no non-trivial subgroups (when $N$ is prime), a fact which fails dramatically for $\bbf_p^n$.
\end{enumerate}

One of the remarkable aspects of Theorem~\ref{thmain} is its quantitative strength. It is an indication of the depth and quality of the ideas introduced by Manners that his proof moves us from no bounds at all to bounds that are, at worst, only doubly exponential in the parameter $\delta$. The presence of the double exponential when $s\geq 4$ is, moreover, largely a technical artefact of one particular recursive argument in the proof. Most of the techniques used are much more efficient, and Manners suspects that a more refined version of this part of the argument should be possible, which would bring the bounds down to exponential in $\delta^{O(1)}$ for all $s$. It is not clear what the true order of magnitude of the bounds should be, and the best lower bounds known are only polynomial in $\delta$.

Our aim is, as stated above, to explain carefully what the inverse theorem says and sketch some of the ideas used by Manners. In limited space there is necessarily a great deal more that must be left unsaid. In particular, we omit any discussion of how the inverse theorem can be applied. The qualitative version has already been put to a variety of uses, and now that Manners has proved an effective quantitative version, its power has only grown. For the reader interested in seeing how inverse results can be applied to number theoretic problems we refer to the work of \textcite{GT10} on linear equations in primes, or the book by \textcite{Tab} on higher order Fourier analysis.

In Section 3 we will define carefully what a nilsequence is, completing the statement of Theorem~\ref{thmain}. Before going into details, however, we give a very rough sketch of how proofs of inverse theorems tend to go. The reader unfamiliar with nilsequences can, for now, read `a nilsequence of degree $s$' as `$e(P(n))$ where $P\in \tfrac{1}{N}\bbz[x]$ is a polynomial of degree $s$'. 

\section{A bird's eye view}
Fix some $\delta >0$ and let $f:\bbz_N\to\bbc$ be a 1-bounded function such that $\norm{f}_{U^{s+1}} \geq \delta$. Our goal is to find some nilsequence $\psi$ of degree $s$ such that $f$ `correlates' with $\psi$, that is, $\abs{\bbe f\overline{\psi}}\gg_\delta 1$.\footnote{Here we use the Vinogradov notation $\gg_\delta 1$ to mean $>c(\delta)$ for some function depending only on $\delta$.} The case $s=1$ has already been proved using Fourier analysis, allowing for an inductive approach.

We first sketch the approach of \textcite{GTZ12}. They first use the identity 
\begin{equation}\label{induc}
\| f\|_{U^{s+1}}^{2^{s+1}} = \bbe_h \| \Delta_h f\|_{U^s}^{2^s}
\end{equation}
to deduce that if $\norm{f}_{U^{s+1}}\geq \delta$ then there are $\gg_\delta N$ many $h$ such that $\norm{\Delta_h f}_{U^s}\gg_\delta 1$, and hence by induction there exist corresponding nilsequences $\psi_h$ of degree $s-1$ which correlate with $\Delta_h f$. By a Cauchy-Schwarz argument similar to one used by \textcite{Go01}, they then show that the function $h \mapsto \psi_h$ is `approximately linear', in that there are $\gg_\delta N^3$ many quadruples $(h_1,h_2,h_3,h_4)$ such that
\begin{equation}\label{applin}h_1+h_2=h_3+h_4\quad\textrm{and}\quad \psi_{h_1}\psi_{h_2}\approx \psi_{h_3}\psi_{h_4}\quad\textrm{(for some notion of $\approx$)}.
\end{equation}
This argument makes crucial use of the cocycle identity (valid for all $f:\bbz_N\to\bbc$ such that $\abs{f(x)}\equiv 1$)
\begin{equation}\label{cocycle}
\Delta_{h+k}f(x)=\Delta_k f(x) \Delta_h f(x+k),
\end{equation}
which is a trivial consequence of the definition of $\Delta$. They then use tools from additive combinatorics to further regularise the structure of the function $h \mapsto \psi_h$, after passing to a smaller set of $h$, so that on these $h$ the map $h\mapsto \psi_h$ is genuinely linear.

After a `symmetrisation' argument, which revolves around the symmetry property
\begin{equation}\label{symmetry}
\Delta_{h}\Delta_kf(x)=\Delta_k \Delta_h f(x),
\end{equation}
one can explicitly write $\psi_h(x)\approx \Delta_h\Psi(x)$, where $\Psi$ is a nilsequence of degree $s$, at least up to terms of degree $s-2$. Since each $\Delta_h f$ correlates with $\psi_h$ and $\psi_h$ agrees with $\Delta_h\Psi$ up to a nilsequence of degree $s-2$, it follows that $\Delta_h (f\overline{\Psi})$ correlates with a degree $s-2$ nilsequence. By the converse to the inverse theorem it follows that $\|\Delta_h(f\overline{\Psi})\|_{U^{s-1}}\gg_\delta 1$. Since this is true for $\gg_\delta N$ many $h$, the recursive definition \eqref{induc} implies that $\| f\overline{\Psi}\|_{U^s}\gg_\delta 1$ and so, by induction, there is a nilsequence $\Psi'$ of degree $s-1$ such that $f\overline{\Psi}$ correlates with $\Psi'$, and hence $f$ correlates with $\Psi\Psi'$, a nilsequence of degree $s$, as required.

Putting all this into practice is highly non-trivial, and requires frequent recourse to technical results on the quantitative equidistribution of polynomial sequences on nilmanifolds developed by \textcite{GT12a}.

Manners takes a different approach, which has more in common with the original paper of \textcite{Go01} (which proved a quantitatively reasonable \emph{local} inverse theorem). Instead of using that for many $h\in \bbz_N$ the derivative $\Delta_h f$ has large $U^s$ norm, then using induction and working with degree $s-1$ nilsequences, Manners differences all the way down to $U^2$, finding many $\bvec{h}\in \bbz_N^{s-2}$ such that $\Delta_{\bvec{h}}f$ has large $U^2$ norm, and then applies Fourier analysis there, before working all the way back up again. 

Applying the $U^2$ inverse theorem to $\Delta_{\bvec{h}}f$ produces a function $\bvec{h}\mapsto \psi_{\bvec{h}}$, say, which finds some linear character with which $\Delta_{\bvec{h}}f$ correlates. Again by repeated use of the Cauchy-Schwarz inequality we may then find some structure within this $\psi$ function. Rather than show that it is approximately a linear function, however, as in \eqref{applin}, we obtain something much weaker, and can only show that it is approximately a polynomial of degree $s$.

One is then faced with the daunting task of understanding the structure of such approximate polynomials for higher degree, rather than just the linear case. This is the principal component of the proof of \textcite{Ma18}: a structural theorem for approximate polynomials of any degree $s\geq 1$. He proves that any such approximate polynomial must resemble a structured algebraic object called a nilpolynomial.

With such an inverse result for approximate polynomials in hand, the inverse result for uniformity norms is (relatively) straightforward to deduce. It is the proof of the inverse result for approximate polynomials that takes up the bulk of the paper of Manners, requiring an argument of extraordinary subtlety and care, especially since we need to keep careful track of all quantitative parameters, and ensure they do not grow out of control. 

We will discuss the proof of the inverse result for approximate polynomials in Sections 5 and 6 of this article, and state it precisely (along with defining what a nilpolynomial is) in Section 4. Before then, however, we return to the inverse theorem itself, and precisely define the concept of a nilsequence.

\section{Nilobjects}\label{nilob}

In this section we will build towards a precise definition of a nilsequence, which are the obstructions to having small uniformity norm that appear in the statement of the inverse theorem. Nilsequences are functions $\psi:\bbz\to \bbc$ that should certainly include the classical (phase) polynomials $x\mapsto e(P(x))$, where $P\in \bbr[x]$, but also need to allow for the possibility of more complicated `bracket' structure, as alluded to in the introduction. 

The reader may wonder why we are discussing nilsequences from $\bbz\to\bbc$, and then obtaining well-defined functions on $\bbz_N$ by insisting on $N$-periodicity, rather than directly defining nilsequences from $\bbz_N\to\bbc$. The latter is possible, using the machinery of Host-Kra cube groups, as described in \textcite[Appendix C]{Ma18}. For this expository article we have adopted a more low-brow approach, which is easier to describe, and simply insist on $N$-periodicity by fiat.

To motivate the construction of nilsequences, we focus on one simple example: the function $n\mapsto e(\alpha n^s)$, which our informal discussion suggests should be a nilsequence of degree $s$. This function can be broken down into three steps, factoring first through $\bbr$ and then the quotient space $\bbr/\bbz$:
\[
\begin{tikzcd}
 n \arrow[r,mapsto] & \alpha n^s \arrow[r,mapsto] & \{\alpha n^s\}\arrow[r,mapsto] & e(\alpha n^s)\\
  \bbz \arrow[r] & \bbr \arrow[r] & \bbr/\bbz \arrow[r] & \bbc.
\end{tikzcd}
\]

This factor construction is how we will define a general nilsequence, but we want to have the flexibility of allowing the intermediate group $\bbr$ to be replaced by some other, possibly non-abelian, group $G$. The first map from $\bbz\to G$ should still be a polynomial of degree $s$ (whatever this means for a general group), and the second should be from this group to some kind of compact quotient space. 

Before we give the full definition of a nilsequence we need to first say exactly what we mean by `a polynomial of degree $s$ from $\bbz\to G$' where $G$ is some (not necessarily abelian) group, and then to say what kind of quotient structure is needed.

\subsection{Polynomial maps}
Intuitively, a polynomial map of degree $s$ is one whose derivatives of order $s+1$ all vanish. This suffices for a definition when, for example, considering maps from $\bbr$ to $\bbr$. We need to generalise the concept to handle maps from arbitrary groups to other arbitrary groups, which means we need to be more careful both about what a `derivative of order $s+1$' means, and, indeed, what it means to `vanish'. The correct generalisation of polynomial maps to arbitrary groups was given by \textcite{Le02}.

To properly define polynomial maps we need to consider not just groups, but groups equipped with an additional structure. A \emph{filtered group} of degree $s$ is a group $G$ equipped with a filtration of subgroups\footnote{Manners calls this a `proper filtration', and allows $G_0\neq G_1$ in a filtration. This has a few technical advantages, but for this exposition we will simplify matters by always assuming $G_0=G_1$.} (closed if we are considering topological groups):
\[G=G_0=G_1\geq G_2\geq \cdots\geq G_s \geq \{e\}=G_{s+1}=G_{s+2}=\cdots\]
such that if $g_i\in G_i$ and $g_j\in G_j$ then $[g_i,g_j]=g_i^{-1}g_j^{-1}g_ig_j\in G_{i+j}$. Observe that a filtered group is, as we have defined it, necessarily nilpotent (of class at most $s$), and note that a filtered group of degree $s$ is also one of any degree $s'\geq s$. If $G$ is abelian then it has a filtration of any degree $s\geq 1$, the standard degree $s$ filtration, where $G_i=G$ for $0\leq i\leq s$, which we denote by $G_{(s)}$.

The most important example is $\bbr_{(s)}$, the standard degree $s$ filtration of the reals. Another illustrative example to bear in mind, and our principal non-abelian example, is the degree $2$ filtration of the Heisenberg group (when we speak of the Heisenberg group it will always be with this filtration) defined by
\[\begin{pmatrix} 1 & \bbr & \bbr \\ 0 & 1 & \bbr \\ 0 & 0 & 1\end{pmatrix}=G_0=G_1\geq G_2=\begin{pmatrix} 1 & 0 & \bbr \\ 0 & 1 & 0 \\ 0 & 0 & 1\end{pmatrix}.\]
We extend the definition of derivative to maps between arbitrary groups by defining $\Delta_hp:H\to G$, for $p:H\to G$ and $h\in H$, by
\[\Delta_h p(x)= p(x)p(hx)^{-1},\]
and when $\bvec{h}\in H^m$ we define $\Delta_{\bvec{h}}p=\Delta_{h_1}\cdots \Delta_{h_m}p$.

\begin{definition}[Polynomial map]
If $H$ and $G$ are both filtered groups then a polynomial map $p:H\to G$ is a continuous map such that for all $x\in H$ and $\bvec{h}\in H^m$, if $h_j\in H_{i_j}$ for $1\leq j\leq m$ then
\[\Delta_{\bvec{h}} p(x)\in G_{i_1+\cdots+i_m}.\]

If $A$ is an abelian group then a polynomial from $H$ to $A$ of degree $s$ is a polynomial map in the above sense from $H$ to $A_{(s)}$, the standard degree $s$ filtration of $A$, when $s\geq 1$, and a constant map when $s=0$. 
\end{definition}

For example, if $G=\bbr_{(s)}$ then this condition is just saying that $\Delta_{h_1,\ldots,h_{s+1}}p(x)=0$ for all $x,h_1,\ldots,h_{s+1}\in H$. This is therefore a generalisation of the standard definition that `a polynomial of degree $s$ is a function whose derivatives of order $s+1$ identically vanish'. Note that between filtered groups we do not need to specify the degree of the polynomial map, since this information is already present in the degree of the filtered groups $H$ and $G$.

For our purposes, the domain $H$ will usually be the filtered group $\bbz_{(1)}$, in which case the definition becomes
\[\Delta_{h_1,\ldots,h_m}p(x)\in G_m\]
for all $x,h_1,\ldots,h_m\in \bbz$. For example, polynomial maps from $\bbz_{(1)}$ to $\bbr_{(s)}$ are exactly those polynomials in $\bbz[x]$ of degree at most $s$, and polynomial maps from $\bbz_{(1)}$ to the Heisenberg group are those of the form
\[n\mapsto \begin{pmatrix} 1 & c_1n+c_2 & c_3n^2+c_4n+c_5 \\ 0 & 1 & c_6n+c_7 \\ 0 & 0 & 1\end{pmatrix}\]
where $c_1,\ldots,c_7\in \bbr$ are any constants. It can be shown that, with the definition we have given, the only polynomial maps from a \emph{finite} group to a torsion-free abelian group are constant maps. This is part of the reason that we need to define nilsequences with domain $\bbz$, rather than $\bbz_N$. As mentioned above, there is a more abstract approach that allows for a non-trivial notion of non-constant polynomial maps from $\bbz_N$ to a nilmanifold (which we will define in the next section), which is that used by Manners, but we will not pursue this here. 

The set of polynomial maps as defined here is, just like the set of classical polynomials, closed under pointwise multiplication and composition. This is not at all obvious from the definition! This remarkable fact was proved by \textcite{Le02}. (In fact the set of polynomial maps can itself  be given the structure of a filtered group.)

\subsection{Nilmanifolds}

Having defined what we mean by a polynomial map from $\bbz$ to $G$, where $G$ is a filtered group of degree $s$, we now need to say what kind of quotient structure is needed for a nilsequence. We would like to have structure sufficient to capture the prototypical example of a nilsequence discussed above, which factors through $\bbz\to \bbr\to \bbr/\bbz$. In general we will replace $\bbr$ by some other $\bbr^d$, given the structure of a (filtered) Lie group, such as the normal additive group structure on $\bbr^d$, or the Heisenberg group. We would then like to quotient out by `the integer part' of $\bbr^d$, in a way which is compatible with the Lie group structure.  This leads us to the following definition of a nilmanifold.

\begin{definition}[Nilmanifold]
Let $s\geq1$ be an integer, and $d_1,\ldots,d_s,M\geq 0$ be any integers. A nilmanifold $G/\Gamma$ of degree $s$, dimension $(d_1,\ldots,d_s)$, and complexity $M$, is the quotient space of a pair $(G,\Gamma)$ where
\begin{itemize}
\item $G$ is a simply connected filtered Lie group of degree $s$, where the filtration subgroup $G_i$ is a real manifold with dimension $d_i+\cdots+d_s$, with `generators' $\gamma_{ij}$ for $1\leq i\leq s$ and $1\leq j\leq d_i$ (i.e. homomorphisms $\gamma_{ij}:\bbr\to G$ where we write $\gamma_{ij}^a=\gamma_{ij}(a)$), 
\item every $g\in G$ has a unique representation as
\[\prod_{1\leq i\leq s}\prod_{1\leq j\leq d_i}\gamma_{ij}^{a_{ij}}\]
for some `coordinates' $a_{ij}\in \bbr$, and the associated bijection $\sigma:G \to \bbr^{d_1+\cdots +d_s}$ is a homeomorphism,
\item the filtration subgroup $G_r$ consists of those $g$ such that $a_{ij}=0$ for all $i<r$ and $1\leq j\leq d_i$,
\item for all $i,j,k,l$, the coordinates of the commutator $[\gamma_{ij},\gamma_{kl}]$ are integers bounded by $M$, and 
\item $\Gamma\leq G$ is the subgroup consisting of all those elements with integral coordinates.
\end{itemize}
We also write $D=d_1+\cdots+d_s$ for the dimension of $G/\Gamma$.
\end{definition}

We caution that $G/\Gamma$ is not necessarily a group, but it is a compact space which has $[0,1)^D$ as a fundamental domain. Heuristically, one should think of a nilmanifold of dimension $D$ as the cube $[0,1)^D$ acted upon by a $D$-dimensional Lie group (although, again, we stress that it is not necessarily actually a group itself). The complexity parameter $M$ allows us to swap coordinates in a representation of an element of $G$ at a `cost' of at most $M$. It offers some control of precisely how non-abelian our nilmanifold is. Furthermore, the coordinate homeomorphism $\sigma: G\to \bbr^D$ allows us to define a metric on $G/\Gamma$ by
\[d(x_1\Gamma,x_2\Gamma) = \inf_{\gamma_1,\gamma_2\in \Gamma}\norm{\sigma(x_1\gamma_1)-\sigma(x_2\gamma_2)}_1,\]
and hence we can speak of Lipschitz functions on $G/\Gamma$.

 This is the definition that Manners uses, which is quite explicit. The usual, more compact, definition, as used by \textcite{GTZ12} for example, is that a nilmanifold is a quotient $G/\Gamma$ where $G$ is any connected, simply connected, nilpotent Lie group, and $\Gamma$ is a discrete subgroup such that $G/\Gamma$ is compact. This can be shown to be equivalent to our explicit definition using the lower central series filtration and the machinery of Mal'cev coordinates, but we prefer to be explicit from the outset. 

The simplest abelian example is the nilmanifold of degree $s$, dimension $(0,\ldots,0,1)$, and complexity $0$ given by $G=\mathbb{R}_{(s)}$ and $\Gamma=\bbz$. A slightly less trivial abelian example is the nilmanifold of degree $s$, dimension $(1,\ldots,1)$, and complexity $0$, where $G=\bbr^s$ is given a filtration of degree $s$ by
\[G_1=\bbr^s\geq \bbr^{s-1}\geq \cdots \geq \bbr \geq \{0\}\]
and $\Gamma=\bbz^s$. The simplest non-abelian example is the Heisenberg filtered group discussed above, which (quotiented out by the obvious integral subgroup) is a nilmanifold of degree $2$, dimension $(2,1)$, and complexity $1$.

\subsection{Nilsequences}
Having defined a general polynomial map and an appropriate quotient object to factor through we can now give the definition of a nilsequence. 
\begin{definition}[Nilsequence]
A nilsequence $\psi:\bbz\to\bbc$ of degree $s$, dimension $D$, complexity $M$, and parameter $K$, is a function that factors as
\[
\begin{tikzcd}
  \bbz \arrow[r,"p"] & G \arrow[r, "\pi"] & G/\Gamma \arrow[r, "F"] & \bbc
\end{tikzcd}
\]
where
\begin{enumerate}
\item $G/\Gamma$ is a nilmanifold of degree $s$, dimension $D$, and complexity $M$,
\item $p:\bbz \to G$ is a polynomial map;
\item $\pi:G\to G/\Gamma$ is the projection map; and
\item $F:G/\Gamma \to \bbc$ is Lipschitz with constant $K$.
\end{enumerate}
We say that $\psi$ is $N$-periodic if $\pi\circ p: \bbz\to G/\Gamma$ is.
\end{definition}

For example, if we take the abelian nilmanifold $G=\bbr^s$ and $\Gamma=\bbz^s$ as above, then $p:\bbz\to \bbr^s$ will have the form
\[p(n) = (p_1(n),\ldots,p_s(n))\]
where $p_j\in \bbr[x]$ is a polynomial of degree $j$.  If we take, for example, the Lipschitz function from $(\bbr/\bbz)^s$ to $\bbc$ defined by $(\alpha_1,\ldots,\alpha_s)\mapsto e(\alpha_1+\cdots+\alpha_s)$, then we see that any function of the form $n\mapsto e(P(n))$, where $P\in \bbr[x]$ is a polynomial of degree $s$, is indeed a nilsequence of degree $s$ as expected.

One would also expect that the more complicated phase polynomials, where we allow the bracket operation also, are nilsequences. There is an unfortunate technical hitch, however, which means they are not precisely nilsequences as we have defined them. Consider, for example, the function $n\mapsto e(-\alpha n\lfloor \beta n\rfloor)$, which we would expect to be a nilsequence of degree $2$. One would like to say it is, using the Heisenberg nilmanifold. Indeed, one can check that the map
\[n\mapsto \begin{pmatrix}1&\alpha n& 0\\ 0&1& \beta n \\ 0&0&1\end{pmatrix} \]
is indeed a polynomial map from $\bbz$ to the Heisenberg group, and the projection map $\pi: G\to G/\Gamma$ for the Heisenberg nilmanifold has the form
\[\begin{pmatrix}1& x& z\\ 0&1& y \\ 0&0&1\end{pmatrix}\mapsto \begin{pmatrix}1& \{x\}& \{z-x\lfloor y\rfloor\}\\ 0&1& \{y\} \\ 0&0&1\end{pmatrix}\]
(because with this choice there is a unique $\gamma\in \Gamma=\begin{psmallmatrix}1&\bbz&\bbz\\ 0&1&\bbz\\0&0&1\end{psmallmatrix}$ such that $g=\pi(g)\gamma$), so the first two steps of the definition work as expected. The problem is that, identifying $G/\Gamma$ with the fundamental domain $[0,1)^3$, the map $F:G/\Gamma\to\bbc$ defined by $F(a,b,c)=e(c)$ is not Lipschitz -- indeed, it is not even continuous! This is because as $y$ varies from $0.99$ to $1.01$ the floor function $\lfloor y\rfloor$ has a discontinuous jump, and hence the $c$ coordinate of 
\[\begin{pmatrix} 1&x&z\\ 0&1&y\\ 0&0&1\end{pmatrix}\]
in the fundamental domain, which is $\{ z-x\lfloor y\rfloor\}$, will jump discontinuously (assuming $x\neq 0$). It follows that $F$ is not a continuous function. In their work on the $U^4$ norm \textcite{GTZ11} allow for $F$ to be discontinuous, so that these bracket phase polynomials are genuine nilsequences, but the ensuing analytic complications are hard to deal with for $s\geq 4$. We therefore use the definition above, where the function $F$ is analytically quite nice, the disadvantage being that the `non-abelian' nilsequences do not quite look how one would expect if one is guided by bracket polynomials as the motivating examples. The reader should view this as a purely technical issue, however, and would not lose any intuition in continuing to view nilsequences as exponentials of bracket polynomials for all intents and purposes.

Finally, we remark that many of the analytic arguments when dealing with nilsequences become much simpler if we insist that $F$ is a smooth function (with control on the associated Sobolev norms), rather than just Lipschitz. This alternative definition will hopefully be adopted in future work on the subject, but almost all of the literature (including the work of Manners) works with $F$ as Lipschitz, and so we have followed convention.

\subsection{Nilsequences as obstructions}\label{secob}

Finally, we sketch a proof that nilsequences are indeed obstructions to having small uniformity norm, which is the (easier) converse to the inverse theorem. This is a trivial repeated differencing procedure when the nilsequence has the simple form $n\mapsto e(P(n))$ for $P(x)\in\bbr[x]$ a genuine polynomial, but the more general case requires some work, and was first proved by \textcite{GT08b}. We sketch an alternative proof given by \textcite[Appendix G]{GTZ11}.

\begin{theorem}[Green-Tao-Ziegler]
Let $\delta>0$. If $f:\bbz_N\to\bbc$ is a $1$-bounded function and there exists a $1$-bounded, $N$-periodic, nilsequence $\psi:\bbz\to\bbc$ of degree $s$, complexity $M$, parameter $K$, and dimension $D$ such that
\[\abs{\bbe_{x} f(x)\overline{\psi(x)}}\geq \delta\]
then there exists $\epsilon>0$, depending only on $(s,\delta,M,K,D)$, such that
\[\norm{f}_{U^{s+1}}>\epsilon.\]
The dependence of $\epsilon$, once we have fixed $s$, is polynomial in all the other parameters.
\end{theorem}
\begin{proof}[Sketch proof]
We use induction on $s$. The case $s=1$ is simple Fourier analysis coupled with the fact that all degree 1 nilsequences of bounded complexity, parameter, and dimension can be approximated by a short trigonometric sum. 

For general $s\geq 2$, suppose the nilsequence factors as $\psi=F\circ \pi\circ p$. Since $G_s/(\Gamma\cap G_s)$ is a compact abelian group we can use Fourier analysis to expand $F$ as the sum of $F_{\chi}$ as $\chi$ ranges over the dual group of $G_s/(\Gamma\cap G_s)$, where 
\[F_{\chi}(x)= \int_{G_s/(\Gamma\cap G_s)} F(xg_s)\overline{\chi(g_s)} \td g_s.\]
Furthermore, each $F_\chi$ has a vertical character, which means that for all $g_s\in G_s/(\Gamma\cap G_s)$ and all $x\in G/\Gamma$ we have $F_\chi(g_sx)=\chi(g_s)F_\chi(x)$. After proving suitable decay properties of this Fourier series, we may find a single character $\chi$ such that
\begin{equation}\label{aux1}
\abs{\bbe_x f(x)\overline{F_\chi\circ \pi\circ p(x)}}\gg_\delta 1.
\end{equation}
We now consider the function $\tilde{F}_h:\bbz_N\to\bbc$, defined for any $h\in \bbz_N$ by
\[\tilde{F}_h(x)= F_\chi(p(x)\Gamma)\overline{F_\chi(p(x+h)\Gamma)}.\]
Squaring \eqref{aux1} and expanding out the brackets, then using the pigeonhole principle, we deduce that there exist $\gg_\delta N$ many $h\in \bbz_N$ such that
\[\abs{\bbe_x \Delta_h f(x)\tilde{F}_h(x)}\gg_\delta 1.\]
One can show (see \textcite[Appendix G]{GTZ11} for details) that $\tilde{F}_h$ is, for any fixed $h$, itself a nilsequence of degree $s-1$ (this is not entirely straightforward, and some care must be taken to ensure that the complexity does not increase by too much). The inductive hypothesis then implies that $\norm{\Delta_hf}_{U^s}\gg_\delta1$ for $\gg_\delta N$ many $h$, and thus
\[\norm{f}_{U^{s+1}}^{2^{s+1}}=\bbe_h \norm{\Delta_h f}_{U^s}^{2^s}\gg_\delta 1\]
as required.
\end{proof}

With all the terms used in the statement of Theorem~\ref{thmain} defined, we now turn to sketching the proof of Manners.

\section{An inverse theorem for approximate polynomials}
The core of Manners' proof is another inverse result, but not for uniformity norms, rather for the more combinatorial notion of an approximate polynomial. We first introduce the additive discrete derivative. Let $H$ be an abelian group, then for any $f:H\to \bbr$ and $h\in H$ we define $\partial f:H\to \bbr$ by
\[\partial_h f(x)= f(x)-f(x+h),\]
again extended to multiple derivatives $\partial_{\bvec{h}}$ in the natural fashion.\footnote{This is just the derivative operator $\Delta$ applied to functions taking values in the additive group of the reals, but we find it helpful to have a reminder of whether we're thinking additively or multiplicatively.} Recall that a polynomial of degree $s$ from $H$ to $\bbr$ is one whose derivatives of order $s+1$ are all identically zero -- that is, a function $f$ such that $\partial_{\bvec{h}}f(x)=0$ for all $(x,\bvec{h})\in H^{s+2}$. To obtain an approximate polynomial we weaken this to just holding for some small proportion of $H^{s+2}$. 
\begin{definition}[Approximate polynomial]
Let $H$ be a finite abelian group. Let $\epsilon>0$ and $X\subset H$. A function $f:X\to \bbr$ is an $\epsilon$-polynomial of degree $s$ on $X$ if
\[\#\{ (x,\mathbf{h})\in H^{s+2}: \partial_{\mathbf{h}}f(x)=0\}\geq \epsilon \abs{H}^{s+2},\]
where the set ranges over those $(x,\bvec{h})$ such that $\partial_{\bvec{h}}f(x)$ is defined (i.e. $x+\omega\cdot \bvec{h}\in X$ for all $\omega\in\{0,1\}^{s+1}$). In particular, it follows that $\abs{X}\gg_s \epsilon \abs{H}$. 
\end{definition}

It is an instructive calculation to verify that all of the following are examples of approximate polynomials (for some $\epsilon>0$ bounded away from zero, that is dependent on the precise construction):
\begin{enumerate}
\item bracket polynomials, such as $x\mapsto \{ mx/N\}$ for some fixed $m\in\bbz$;
\item any function taking any $K$ arbitrary values (where $K$ is independent of $H$, and $\epsilon$ can depend on $K$);
\item a function that agrees with a bracket polynomial on a positive density set and is otherwise random; and
\item a function which switches between a few different bracket polynomials as we move across some arbitrary partition of $\bbz_N$.
\end{enumerate}

The inverse result mentioned above, and the keystone of Manners' work, is the fact that examples of these kinds are, essentially, the only types of approximate polynomials. Before stating it, we will define the concept of nilpolynomials, which are to bracket polynomials as nilsequences are to phase bracket polynomials. Again, this is because they are technically far more convenient to work with, although they are essentially equivalent, and the reader can think of them as bracket polynomials if they prefer. 

\subsection{Nilpolynomials}

We need to work in a slightly more general context than just cyclic groups, so we give the definition for $\bbz^d$ for any $d\geq 1$.
\begin{definition}[Nilpolynomial]
A nilpolynomial $f:\bbz^d\to\bbr$ of degree $s$, dimension $D$, and complexity $M$, is a function that factors as
\[
\begin{tikzcd}
  \bbz^d \arrow[r,"\rho"] & G \arrow[r, "F"] & \bbr
\end{tikzcd}
\]
where
\begin{enumerate}
\item there is $\Gamma\leq G$ such that $G/\Gamma$ is a nilmanifold of degree $s$, dimension $D$, and complexity $M$, 
\item there is some polynomial map $p:\bbz^d\to G$ such that $\pi\circ p=\pi\circ \rho$, where $\pi$ is the usual projection map $\pi:G\to G/\Gamma$,
\item $\rho$ maps into the ball of radius $M$ around the identity, and
\item $F:G\to \bbr$ is a polynomial map of degree $s$. 
\end{enumerate}
We say that $f$ is $N$-periodic if the function $\rho$ is (so that the nilpolynomial induces a function $f:\bbz_N^d\to \bbr$).
\end{definition}
The point is that we are not just limiting ourselves to compositions $F\circ p$, where $p:\bbz^d\to G$ and $F:G\to \bbr$ are polynomials, but we also allow the flexibility of combining $p$ with some arbitrary (but bounded) `lift' from $G/\Gamma$ to $G$. Roughly speaking, therefore, nilpolynomials are `polynomials factored through nilmanifolds', although we stress that they are \emph{not} (necessarily) polynomial maps from $\bbz^d$ to $\bbr$ themselves. If one develops more machinery, as Manners does, one can define a notion of polynomial maps from $\bbz_N^d$ to nilmanifolds directly, but for our purposes, the definition of nilpolynomial that results is equivalent to what we have presented here.

For example, if $G$ is the abelian nilmanifold $\bbr^s$, taking $\rho(n)=(0,\ldots,0,\{ p(n)\})$ and $F:\bbr^s\to\bbr$ as the projection onto the last coordinate, where $p(n)$ is a polynomial of degree $s$, exhibits $n\mapsto \{p(n)\}$ as a nilpolynomial of degree $s$, as expected. It is $N$-periodic if $p(x)\in \tfrac{1}{N}\bbz[x]$. Furthermore, the map $n\mapsto \{p(n)\}+r(n)$ is also a nilpolynomial, now of complexity $M$, for any arbitrary function $r:\bbz\to\bbz$ taking values in $(-M,M)$ (and is still $N$-periodic if $r$ is). More general bracket polynomials, such as $n\mapsto \{\alpha n\}\{\beta n\}$, are also nilpolynomials, and can be constructed through Heisenberg-type nilmanifolds as in the previous section. 

The point of the parameter $M$, and in allowing $\rho$ to take values outside the fundamental domain of $G/\Gamma$, is to give greater flexibility to the function $\rho$ so that it does not itself need to exhibit strict polynomial behaviour, even though $\pi\circ \rho$ does. This means that we can construct functions that take on a bounded number of arbitrary values as genuine nilpolynomials in their own right.

For example, suppose that $A_1\sqcup\cdots \sqcup A_r$ is a partition of $\bbz$, and $f$ takes the value $c_i\in \bbr$ on $A_i$. Let $G=\bbr^r$ with the standard degree 1 filtration, and let
\[\rho(x)= (1_{A_1}(x),\ldots,1_{A_r}(x)).\]
If we choose $F:\bbr^r\to \bbr$ to be the affine projection 
\[(x_1,\ldots,x_r)\mapsto c_1x_1+\cdots+c_rx_r,\]
then $F\circ \rho=f$ as required. This exhibits $f$ as a nilpolynomial of degree 1, dimension $r$, and complexity $1$. 

\subsection{The inverse result for approximate polynomials}
One can show that any nilpolynomial is an approximate polynomial. Manners proves the following strong inverse theorem.

\begin{theorem}[Every approximate polynomial is approximately a nilpolynomial]\label{approx}
If $X\subset \bbz_N^d$ and $f:X\to\bbr$ is an $\epsilon$-polynomial of degree $s$ then there is an $N$-periodic nilpolynomial $g:\bbz^d\to\bbr$ of degree $s$, dimension $O(\epsilon^{-O(1)})$, and complexity $M$ such that
\[f(x) = g(x)\quad\textrm{ for at least }(\epsilon/2)^{O(1)}\abs{X}\textrm{ many }x\in X\]
and 
\[M\leq \begin{cases} \exp(O(\epsilon^{-O(1)}))&\textrm{ if }s\leq 3\textrm{ and }\\ \exp(\exp(O(\epsilon^{-O(1)})))&\textrm{ if }s\geq 4.\end{cases}\]
\end{theorem}

We now sketch how Manners deduces the inverse theorem for uniformity norms from this inverse theorem for approximate polynomials. After this, we will focus on the proof of Theorem~\ref{approx}, and make no further mention of uniformity norms.

\begin{proof}[Sketch proof of Theorem~\ref{thmain} from Theorem~\ref{approx}.]
This part of the proof of Manners uses Fourier analysis combined with elaborate but elementary arguments involving repeated use of the Cauchy-Schwarz inequality. In this it has more in common with the work of \textcite{Go01} than the work of \textcite{GTZ12}, which replaces the use of Fourier analysis with a far more complicated generalised theory of nilsequences of higher degree. This is inevitable in their approach as they have much less combinatorial information about the functions they are working with, and hence have to use as an inductive hypothesis the inverse theorem for the $U^s$ norm quite early in the argument. This produces an ensemble of nilsequences of degree $s$, which must then be explicitly stitched together somehow to form a nilsequence of degree $s+1$.

Instead of using the fact that $\norm{f}_{U^{s+1}}^{2^{s+1}}=\bbe_h \norm{\Delta_h f}_{U^s}^{2^s}$ to relate the $U^{s+1}$ norm of $f$ to $U^s$-type information for the 1-variable function $\Delta_h f$, Manners uses the relationship
\[\norm{f}_{U^{s+1}}^{2^{s+1}} = \bbe_{\bvec{h}\in \bbz_N^{s-1}} \abs{\bbe_x \Delta_{\bvec{h}}f(x)}^2\]
to relate the $U^{s+1}$ norm of $f$ to $U^2$-type information for the $(s-1)$-variable function
\[S_s f(\bvec{h})=\bbe_x \Delta_{\bvec{h}}f(x).\]
We can then take the Fourier transform in one of the $h$ variables and apply Parseval's identity to find many $\bvec{h}\in \bbz_N^{s-2}$ such that there exists some $\chi_{\bvec{h}}\in \widehat{\bbz_N}$ with
\[\abs{\bbe_{h} S_s(\bvec{h},h)\overline{\chi_{\bvec{h}}(h)}}\gg_\delta 1.\]
Repeated use of the Cauchy-Schwarz inequality then implies that the function $\Phi: \bbz_N^{s-2}\to \widehat{\bbz_N}$ which selects this large character $\chi_{\bvec{h}}$ for each $\bvec{h}$ is an approximate polynomial, and we can apply Theorem~\ref{approx} to find some nilpolynomial $P$ which agrees with $\Phi$ for many $\bvec{h}\in \bbz_N^{s-2}$.

There are some serious difficulties at this point arising from the fact that the set of $\bvec{h}$ where $\Phi$ agrees with $P$ is possibly quite sparse, which Manners overcomes by an elegant argument involving a partition of unity and arguing with a number of nilpolynomials at once. Ignoring such issues, and using the symmetry of $S_sf$, we can show that there exists some nilpolynomial $P$ such that
\[\abs{\bbe_{\bvec{h}\in \bbz_N^{s-1}} S_sf(\bvec{h})\overline{e(P(\bvec{h}))}}\gg_\delta 1.\]
We have found correlation of the $(s-1)$-variable function $S_sf$ with a function that resembles a multi-variable nilsequence. It remains to deduce correlation of $f$ itself with a nilsequence. This is done by projecting $e(P(h_1,\ldots,h_{s-1}))$ down to a 1-variable function by replacing it by
\[x\mapsto e(P(h_1+x,\cdots,h_{s-1}+x))\]
for a randomly chosen $\bvec{h}$. Further combinatorial arguments (again relying heavily on the Cauchy-Schwarz inequality) imply that $f$ itself correlates with (essentially) this function, and using the nilpolynomial structure of $P$ one can show that this 1-variable function is a nilsequence, thus completing the proof of Theorem~\ref{thmain}. 

While we have not discussed the behaviour of quantitative factors in this proof, the lack of iterative arguments and reliance on purely elementary arguments involving the Cauchy-Schwarz inequality should reassure the reader that the quantitative strength of Theorem~\ref{approx} is transferred relatively undiminished to Theorem~\ref{thmain}.
\end{proof}


\section{From 1\% structure to 99\% structure}
The rest of this article will discuss the proof of the inverse theorem for approximate polynomials, Theorem~\ref{approx}. In a high-level overview of an argument like this that uses many quantitative parameters it is convenient to use the 1\% and 99\% heuristic language. This means that we will describe a parameter as being 1\% if it is something like $\epsilon>0$, where $\epsilon$ is to be thought of as very small, and as being 99\% if it is more on the scale of $1-\epsilon$. We will continue to use precise notation in some places, but will not hesitate in using this rougher language where it clarifies the exposition, and refer the reader interested in seeing precise versions of what follows to \textcite{Ma18}. 

The ideas employed in what follows are valid in any finite abelian group, so we will use $H$ to denote an arbitrary finite abelian group (which, for our applications, will be $\bbz_N^d$ for some $d\leq s$, but that need not concern us here). Unless otherwise specified, $X$ denotes some subset of $H$ of 1\% density (i.e. $\abs{X}\geq \epsilon\abs{H}$ for some small $\epsilon>0$).

We are given some approximate polynomial $f:X\to \bbr$ of degree $s$, so that the function $\partial_{\bvec{h}}f(x)$ vanishes for at least 1\% of $(x,\bvec{h})\in H^{s+2}$. Our eventual goal is to show that there is a nilpolynomial $g$ such that $g(x)=f(x)$ for 1\% of $x\in H$. This will be done in three stages:
\begin{enumerate}
\item As $f$ only has a 1\% type of structure we will first need to find some $f'$ which has a 99\% type of structure, such that $f(x)=f'(x)$ for 1\% of all $x\in H$;
\item We then show that a function which has this 99\% type of structure must have some weak algebraic structure;
\item Finally, we show that a function with this weak algebraic structure must agree with a genuine nilpolynomial 1\% of the time (and that this 1\% can be arranged to be a subset of the 1\% of values from (1), so that $f$ itself agrees with a genuine nilpolynomial 1\% of the time).
\end{enumerate}

In this section we discuss the first of these steps. Heuristically, we aim to find some $f'$ which agrees with $f$ for a 1\% proportion of $H$, but such that $f'$ itself is a 99\%-approximate polynomial. This cannot be established directly, but we will show that it is true, aside from possible obstructions of lower degree. We can then use the same argument to show that these obstructions themselves enjoy a 99\% type of structure, modulo further obstructions of even lower degree, and so on. This cascading type of structure is captured by the notion of a `polynomial hierarchy', and we will give the definition of this before stating the main result of this section.  

\subsection{Polynomial Hierarchies}
We will be working with derivatives of different orders, for which the following language will be useful. A cube of dimension $k$ is formally just an element of $H^{k+1}$, but we will generally write it as $c=(x,\bvec{h})$ for some $x\in H$ and $\bvec{h}\in H^k$. This should be thought of as a fixed basepoint $x$ together with a $k$-dimensional direction vector $\bvec{h}$. The coordinates of a cube are the $h_1,\ldots,h_k$ components, and the $x$ component is referred to as the basepoint. A $d$-dimensional face of a cube of dimension $k$ is a $d$-dimensional cube obtained by removing any $k-d$ of the directional components and replacing the basepoint $x$ by $x+\sum_{i\in I}h_i$ for some subset $I$ of the removed $k-d$ components. When $c=(x,\bvec{h})$ we will write
\[\partial f(c)=\partial_{\bvec{h}}f(x),\]
so that in particular $\partial f$ is a function on a set of cubes, the dimension of which is encoded in the cube $c$ itself. Writing out the definition in full, we have
\[\partial f(c)=\sum_{\omega\in \{0,1\}^k}(-1)^{\lvert \omega\rvert}f(x+\omega\cdot \bvec{h}).\]
Note that if $c'$ is obtained from $c$ by a permutation of the (directional) coordinates then $\partial f(c)=\partial f(c')$. Furthermore, if $c'$ is obtained from $c$ by a `coordinate reflection', where $x\mapsto x+h_i$ and $h_i\mapsto -h_i$, then $\partial f(c')=-\partial f(c)$. This means that if we are considering the vanishing of $\partial f$ on some set of cubes then it is natural to ask that such a set be closed under coordinate permutations and reflections. 

Since we will also be moving between taking derivatives of different orders for different sets of cubes it is desirable to have some natural compatibility conditions between the sets of cubes at different levels. These conditions are captured by the following definition. 
\begin{definition}[System of cubes]
A system of cubes of degree $s$ and density $\delta>0$ is a family $(S_0,\ldots,S_{s+1})$ of non-empty sets of cubes, with $S_i\subset H^{i+1}$, such that 
\begin{enumerate}
\item each level is symmetric, i.e. each $S_k$ is closed under permutations of the (directional) coordinates and under coordinate reflections, where $x\mapsto x+h_i$ and $h_i\mapsto -h_i$,
\item the system is closed under taking faces, so that if $c\in S_k$ and $c'$ is a face of $c$ of dimension $d$ then $c'\in S_d$, and 
\item for $0\leq k\leq s$, any $(x,\bvec{h})\in S_k$ can be extended in at least $\delta \abs{H}$ many ways to some $(x,\bvec{h},h_{k+1})\in S_{k+1}$.
\end{enumerate}
\end{definition}
Note that the second and third conditions ensure that $\abs{S_0}\geq \delta \abs{H}$, and hence by induction $\abs{S_k}\geq \delta^{k+1}\abs{H}^{k+1}$ for $0\leq k\leq s+1$. In our applications $\delta$ will usually be a 99\% parameter (i.e. very close to $1$) so a system of cubes will consist of 99\% of all possible cubes.

The concept of a `derivatives condition' underpins everything which follows, as it allows us to talk of the derivatives of some function being expressible in terms of simpler functions. For comparison, note that by the definition of derivative, namely
\[\partial g(c)=\sum_{\omega\in \{0,1\}^k}(-1)^{\lvert \omega\rvert}g(x+\omega\cdot \bvec{h}),\]
we know that $\partial g$ can be trivially expressed as a sum involving $g$ and the projected cubes $x+\omega\cdot \bvec{h}$. The $(\bvec{f},k,t,M)$-derivatives condition says that the $k$th derivative of $g$ can instead be expressed as a `low-complexity' (measured by the parameter $M$) linear combination of functions from $\bvec{f}$ -- moreover, only those of level $\leq t-1$. In particular, if $t=0$ this means that $\partial g\equiv 0$ on $S$. More precisely, we have the following.
\begin{definition}[Derivatives condition]
Let $s\geq 1$, let $(d_1,\ldots,d_s)$ be positive integers, and let $k$ and $t$ be integers such that $0\leq k,t\leq s+1$. Let $X\subset H$ and $S\subset H^{k+1}$ be a set of cubes such that $x+\omega\cdot \bvec{h}\in X$ for every $(x,\bvec{h})\in S$ and $\omega\in \{0,1\}^k$. Let $M\geq 1$ be an integer. Let $\bvec{f}=(f_i:X\to\mathbb{R}^{d_i})_{0\leq i\leq s}$ be a family of functions and let $b:S\times \{0,1\}^k\to \bbz^{d_0+\cdots+d_{t-1}}$ be a function.

A function $g:X\to\mathbb{R}$ satisfies the $(\bvec{f},k,t,M)$-derivatives condition on $S$ with coefficients $b$ if $\norm{b(c,\omega)}_1\leq M$ for all $c\in S$ and $\omega\in\{0,1\}^k$, and for all $c=(x,\bvec{h})\in S$
 \[\partial g(c)=\sum_{\omega\in\{0,1\}^k}(-1)^{\lvert \omega\rvert} b(c,\omega)\cdot \bvec{f}_{\leq t-1}(x+\omega\cdot \bvec{h}).\]
 Here $\abs{\omega}=\omega_1+\cdots+\omega_k$ is the Hamming weight of $\omega$, and $\bvec{f}_{\leq {t-1}}$ is the family of functions  $(f_i:X\to\mathbb{R}^{d_i})_{0\leq i\leq t-1}$.
\end{definition}

In isolation, this definition has very little content -- indeed, every $g$ satisfies the $(\bvec{f},k,1,1)$-derivatives condition for any $k\geq 1$ if we let $f_0=g$. Its importance is that it lets us talk about moving `between levels', so that we can express the derivative of $g$ in terms of simpler functions. It is important to note that we allow the function $b(c,\omega)$ to depend in some completely unspecified fashion on $c$ and $\omega$, and indeed its dependence may be quite unwieldy (which is a complication we will have to deal with later on). Crucially, however, it only takes bounded integer values, and so it is simpler in this sense than the less-controlled $\partial g:S\to \bbr$.

The following examples are instructive. Consider first the bracket linear function $g:\bbz_N\to\bbr$ defined, for some fixed $a\in \bbz$, by 
\[g(x) = \{ ax/N\}.\]
This is an approximate polynomial of degree 1. It is easy to check that $g(x+y)=g(x)+g(y)-b_{x,y}$ where $b_{x,y}\in\{0,1\}$ (in fact, $b_{x,y}=1$ if and only if $\{ax/N\}+\{ay/N\}\geq 1$). It follows that, for any $c=(x,h_1,h_2)$,
\begin{align*}
\partial g(c) &= g(x)-g(x+h_1)-g(x+h_2)+g(x+h_1+h_2)\\
&=b_{x,h_2}-b_{x+h_1,h_2}\\
&=b(c)
\end{align*}
where $b(c)\in \bbz$ and $\abs{b(c)}\leq 1$. In particular, the second derivative of $g$ is always a bounded integer. We do not have much control on how this integer depends on $c$, but the key point is that the derivative of $g$ can be expressed as an integer linear combination of `simpler' functions (in this case, just the constant function $\equiv 1$). 

We now consider a more complicated example, the bracket quadratic function $g:\mathbb{Z}_N\to \mathbb{R}$ defined, for some fixed $a,b\in\bbz$, by
\[g(x) = \{ ax/N\}\{bx/N\},\]
which is an approximate polynomial of degree $2$. An explicit calculation, similar to the above, shows that, for any $c=(x,h_1,h_2,h_3)$, 
\[\partial g(c)= \sum_{\omega\in\{0,1\}^3}(-1)^{\lvert \omega\rvert}b(c,\omega)\brac{1+\{ax/N\}+\{bx/N\}},\]
where $b(c,\omega)\in \bbz$ and $\abs{b(c,\omega)}\leq 2$. In particular, the derivative of the quadratic-type function $g$, although it does not vanish completely like that of a genuine quadratic, can be expressed as a simple linear combination of the linear-type functions $1$, $\{ax/N\}$, and $\{bx/N\}$. We encourage the reader to bear these examples in mind for what follows.

Having expressed the derivative of $g$ in terms of simpler functions of a lower level, it is natural to attempt the same with these simpler functions. By applying the derivatives condition inductively, so the derivatives of these functions can themselves be simplified, and so on, we arrive at the following key definition.

\begin{definition}[Polynomial hierarchy]
Let $X\subset H$ and $(S_0,\ldots,S_{s+1})$ be a system of cubes such that, for $0\leq i\leq s+1$, we have $x+\omega\cdot\bvec{h}\in X$ for every $(x,\bvec{h})\in S_i$ and $\omega\in \{0,1\}^i$.

A system of functions $\bvec{f}=(f_i:X\to \mathbb{R}^{d_i})_{0\leq i\leq s}$ is a polynomial hierarchy of height $s$ and complexity $M$ on $S_0,\ldots,S_{s+1}$ if, for any $0\leq i\leq s$, the function $f_i$ satisfies the $(\bvec{f}, i+1, i,M)$-derivatives condition on $S_{i+1}$. 

Furthermore, we say that a function $g:X\to\mathbb{R}$ sits at the top of a $(D,M)$-polynomial hierarchy of height $s$ on $(S_0,\ldots,S_{s+1})$ if there is a polynomial hierarchy $\bvec{f}$ of height $s-1$ and complexity $M$ on $(S_0,\ldots,S_s)$ with $d_0+\cdots+d_s=D$ such that $g$ satisfies the $(\bvec{f},s+1,s,M)$-derivatives condition on $S_{s+1}$. 
\end{definition}
Technically we have only defined the derivatives condition for functions $g:X\to\mathbb{R}$, but this is easily generalised to functions $g:X\to\mathbb{R}^{d_i}$ by interpreting it as holding for all $d_i$ individual component functions (where we allow different coordinate functions for each $1\leq j\leq d_i$).  Note that as above the condition at $i=0$ is equivalent to $\partial f_0(c)=0$ for all $c\in S_1$, which (provided $\abs{S_1}> \tfrac{1}{2}\abs{H}^2$) is equivalent to $f_0:X\to \bbr^{d_0}$ being a constant function.

Roughly speaking, this definition is saying that all the functions on level $i$ can be reduced to functions on level $i-1$ after taking derivatives of order $i+1$. In applications we will think of functions on level $i$ as resembling polynomials of degree $i$. This definition may then seem odd, since then one would expect to reduce to level $i-1$ after only a single derivative. 

The way to correct one's intuition is to realise that the relationship between level $i$ and the lower levels of a polynomial hierarchy does {\em not} correspond to that between polynomials of degree $i$ and \emph{all} polynomials of degree $\leq i-1$. Instead, the levels $\leq i-1$ are collections of obstructions that capture how objects on level $i$ {\em fail} to be polynomials of degree $i$. Put another way, if a function $f$ on level $i$ is meant to resemble a polynomial of degree $i$ then we expect $\partial^{i+1} f=0$. The functions in $f_{\leq i-1}$ capture to what extent this fails. We now give a more easily remembered `slogan form' of the definition, summarising this discussion. 

\begin{definition}[Polynomial hierarchy, slogan form]
A system of functions $\bvec{f}=(f_i: X\to \bbr^{d_i})_{0\leq i\leq s}$ is a polynomial hierarchy if functions on level $i$ behave like polynomials of degree $i$ up to errors which are a bounded linear combination of functions on levels $\leq i-1$.
\end{definition}

\subsection{Building a polynomial hierarchy}
The following is the first major step in the proof of Theorem~\ref{approx}. To avoid obscuring the content under technicalities we state it in rough form only. For the full, precise, statement, see \textcite[Corollary 2.5.4]{Ma18}.
\begin{proposition}\label{proppoly}
Suppose that $f:X\to\mathbb{R}$ is a 1\%-approximate polynomial of degree $s$. There exists a system of cubes $(S_0,\ldots,S_{s+1})$ which are 99\% dense (i.e.\ of density $\delta=1-\epsilon$ for some small $\epsilon>0$) and $g:S_0\to\mathbb{R}$ such that 
\begin{itemize}
\item $g$ agrees with $f$ 1\% of the time and
\item $g$ sits at the top of a polynomial hierarchy of height $s$ on $S_0,\ldots,S_{s+1}$.
\end{itemize}
\end{proposition}

Recalling the definition of a polynomial hierarchy, this proposition allows to pass from knowing that the $(s+1)$-derivative of $f$ vanishes 1\% of the time, to knowing that the $(s+1)$-derivative of $g$ vanishes 99\% of the time, modulo a collection of functions whose $s$-derivatives vanish 99\% of the time, and so on, until we reach constant functions at the bottom of the hierarchy.

We will see how to make use of this hierarchical structure in the next section. For the remainder of this section, we will sketch how Proposition~\ref{proppoly} is proved. The main ingredient is a recursive application of the following lemma, which says that a 1\%-approximate polynomial of degree $s$ agrees 1\% of the time with a function that is a 99\%-approximate polynomial of degree $s$, up to obstructions which are a bounded complexity combination of degree $\leq s-1$ approximate polynomials. To get a flavour of the kind of difficulties involved the reader may find it instructive to calculate what $\partial_{h_1,h_2}f(x)$ is when $f(x)=\{\alpha x\}\{\beta x\}$. 

\begin{lemma}
Let $f:X\to \bbr$ be a 1\%-approximate polynomial of degree $s$. There is a function $f'$ defined on 99\% of $H$ such that
\begin{enumerate}
\item $f(x)=f'(x)$ for 1\% of $x\in X$;
\item there exists a collection of 1\%-approximate polynomials $g_1,\ldots,g_K$ of degree $s-1$ (for some reasonably bounded $K$); and
\item there exists a function $b: S\times \{0,1\}^{s+1}\to \bbz^K$, with $\norm{b(c,\omega)}_1$ reasonably bounded for each $c\in S$ and $\omega\in \{0,1\}^{s+1}$, such that for 99\% of $s+1$-dimensional cubes $c=(x;\bvec{h})\in H^{s+2}$
\[\partial f'(c) = \sum_{\omega\in \{0,1\}^{s+1}}(-1)^{\abs{\omega}}b(c,\omega)\cdot \bvec{g}(x+\omega\cdot\bvec{h}).\]
\end{enumerate}
\end{lemma}
\begin{proof}[Sketch proof]
Our starting observation is that if $f$ is a 1\%-approximate polynomial of degree $s$ then, for at least 1\% of $h\in H$, the derivative $\partial_h f$ is a 1\%-approximate polynomial of degree $s-1$. Indeed, this is a simple consequence of the pigeonhole principle, averaging over those cubes on which the derivative of $f$ vanishes.

This suggests that we should take our auxiliary functions $g_i$ to be derivatives of $f$. If we define $g_{a,b}$ by 
\[x\mapsto f(x+a)-f(x+b)=\partial_{a-b}f(x+a)\]
then, for suitable choices of $a$ and $b$ (so that $a-b$ is one of the `good' $h$ from the previous paragraph), this is a 1\%-approximate polynomial. This is tricky, however, since this will only be defined when both $x+a$ and $x+b$ are in $X$. Overcoming this difficulty is a subtle matter, and requires careful tracking of the pairs $a,b$ and some initial combinatorial pruning of the set of cubes on which the derivative of $f$ vanishes. We will not address these complications further here, and refer to Section 2 of \textcite{Ma18}.

To define $f'$ we must somehow extend $X$, the domain of $f$ which is 1\% of the group, to a domain which is 99\% of the group. For this we use a random sumset construction. Namely, if $A\subset H$ is a random subset, then we expect the difference set $X-A$ to expand quickly, since there will be few incidences of the form $x_1+a_1=x_2+a_2$. In particular, we can ensure that $X-A$ is 99\% of the entire group by taking $A$ to be a random set of size $O(1)$. We will take the $g_i$ to be the 1\%-approximate polynomials of degree $s-1$ of the form $g_{a,b}$ where $a,b\in A$ (note that there are $O(1)$ such $g_i$).

We then would like to define $f':X-A\to \bbr$ by extending $f$ in the natural way, so that if $x\in X-A$ then $f'(x)=f(x+a)$, where $a\in A$ is chosen so that $x+a\in X$. The problem is that, in general, there will be many such $a$ (still $O(1)$, but certainly more than 1), and this is not an unambiguous definition of $f'$. We define $f'$ by choosing arbitrarily an appropriate $a\in A$ for each $y\in X-A$, and note that the difference between possible choices, namely $f(x+a)-f(x+b)$, is precisely the function $g_{a,b}$. In particular, the arbitrary nature of our choice only affects the definition of $f'$ `up to lower order terms'. 

Finally, after suitable combinatorial pruning, one can show that for 99\% of all possible cubes $c=(x,\bvec{h})\in H^{s+2}$ there exists some shift $c'=(x+a,\bvec{h}+\bvec{t})$ where $a,t_1,\ldots,t_{s+1}$ all lie in the random set $A$, and $\partial f(c')$ vanishes, since $\partial f$ vanishes for at least 1\% of all possible cubes (this is a multi-dimensional version of the fact alluded to above that if $X$ has density 1\% then, provided $A$ is a random set of large but constant size, we expect $X-A$ to have density 99\%). As above, by construction of $f'$ and the auxiliary $g_{a,b}$, one can then show that $\partial f'(c)=\partial f(c')$ up to a `lower order error' which is a bounded linear combination of the $g_{a,b}$ functions. Since $\partial f(c')=0$, however, $\partial f'(c)$ itself is such a linear combination, and we are done.
\end{proof}

\section{From 99\% structure to algebraic structure}

We will now discuss how to pass from the 99\% kind of structure of a polynomial hierarchy, which is inherently combinatorial, to a more algebraic kind of structure, and in particular agreeing often with a genuine nilpolynomial. Recall that we left the previous section with a function $g$ sitting at the top of a polynomial hierarchy of height $s$, denoted by $\bvec{f}$, which means that we can write the derivatives of functions on level $t$ of the hierarchy at cubes of dimension $t$ as
\[\partial f_t(c) = \sum_{\omega\in \{0,1\}^t} b(c,\omega)\cdot \bvec{f}_{\leq t-1}(x+\omega\cdot \bvec{h}).\]
The $b(c,\omega)$ here are some coefficients which, aside from being integer vectors of bounded size, we do not know much about. In particular, we have no information about how they vary with $c$ and $\omega$. 

Suppose that the coefficients $b(c,\omega)$ are independent of $c$. Then whenever $f_t$ is on level $t$ we can write $\partial f_t$ as a simple linear combination of functions of the shape
\[\partial f_t(x,\bvec{h})=\sum_{\omega\in \{0,1\}^t} b_\omega f_{t-1}(x+\omega\cdot \bvec{h}).\]
In particular, the derivatives of order one of the functions on level $t=0$ vanish, and hence these functions are constant, and thus the functions on level $1$ have constant derivatives, and hence must be genuine polynomials of degree $1$. Propagating upwards, we find that $g$ itself must be a genuine polynomial of degree $s$, which is exactly the kind of inverse statement that we are trying to prove. 

Sadly, we cannot expect that the coordinates $b(c,\omega)$ are always this simple. It is helpful to recall the earlier example of a bracket linear (or bracket quadratic function), where the coordinate functions are given by linear combinations of `carry bit' functions. For example, if $f(x)=\{ ax/N\}$ then 
\[\partial f(x,h_1,h_2)= b_{x,h_2}-b_{x+h_1,h_2}\]
where $b_{x,y}=1_{\{ ax/N\}+\{ay/N\}\geq 1}$.

The strategy of Manners is to show that, in general, the functions $b(c,\omega)$ behave structurally like the carry bit functions of bracket polynomial examples. Having done so, using the fact that the functions on level $t=0$ are constant, one can propagate upwards through the levels, showing that the functions on each level resemble a bracket polynomial, eventually arriving at $g$ on the top level.

We will focus, therefore, on the structure of the coefficient functions in the polynomial hierarchy produced in the previous section, and showing that they behave similarly to the carry bit coefficients of the bracket polynomial examples. The key property of carry bits that we require Manners calls a `generalised cocycle' condition.

There are, overall, three main steps to complete the proof of Theorem~\ref{approx}: one needs to show that 
\begin{itemize}
\item if $g$ sits at the top of a polynomial hierarchy of height $s$ then (after some quantitative loss) not only does $g$ satisfy a $(\bvec{f},s+1,s,M)$-derivatives condition, but it also satisfies a stronger derivatives condition where the coefficients $b$ themselves satisfy a `generalised cocycle' condition;
\item all functions satisfying a generalised cocycle condition arise in a similar fashion to carry bits (i.e. as the `integral residue' of derivatives of functions whose derivatives vanish modulo $\bbz$);
\item the polynomials whose carry bits produce the coefficients $b$ can be reconstructed from $b$, and the fact that $g$ sits at the top of a polynomial hierarchy with coefficients $b$ means we can propagate this polynomial construction through the hierarchy to construct a genuine nilpolynomial which agrees with $g$ 1\% of the time. 
\end{itemize}

While none of this is easy or straightforward, the first part is the most difficult -- indeed, it is the most difficult part of the entire proof of \textcite{Ma18}, and is where the most significant quantitative loss occurs. The proofs of the second two parts are more algebraic, especially the third, which is where explicit nilpolynomials have to be constructed. These proofs are long and technical, with many new difficulties arising that Manners has to overcome, but they also have much in common with the kind of explicit nilmanifold constructions that have appeared before in the literature.

To keep the length of our exposition bounded, therefore, we will not discuss the proofs of the latter two points any further. This affords us space to focus on the new \emph{combinatorial} ideas introduced by Manners, and in particular the definition of a `generalised cocycle', which is a generalisation of the cocycle condition \eqref{cocycle} heavily used in previous approaches to the inverse theorem, and which we expect to play an important role in the future of the field.

\subsection{Cocycles}
We begin by recalling the cocycle property of the derivative \eqref{cocycle}. We can rewrite this using the cube notation. For example, in dimension 2, if $c=(x,h_1,h_2)$ is a cube and $u\in H$ then
\begin{align*}
\partial f(c) 
&= f(x)-f(x+h_1)-f(x+h_2)+f(x+h_1+h_2)\\
& = \brac{f(x)-f(x+h_1)-f(x+h_2+u)+f(x+h_1+h_2+u)} \\
&\,\,\,- \brac{f(x+h_2)-f(x+h_1+h_2)-f(x+h_2+u)+f(x+h_1+h_2+u)}\\
&= \partial f(c^u)- \partial f(c_u)
\end{align*}
where $c^u=(x,h_1,h_2+u)$ and $c_u=(x+h_2,h_1,u)$. The relationship between the cubes $c,c^u,c_u$ can be viewed geometrically as seen in Figure~\ref{figco}.

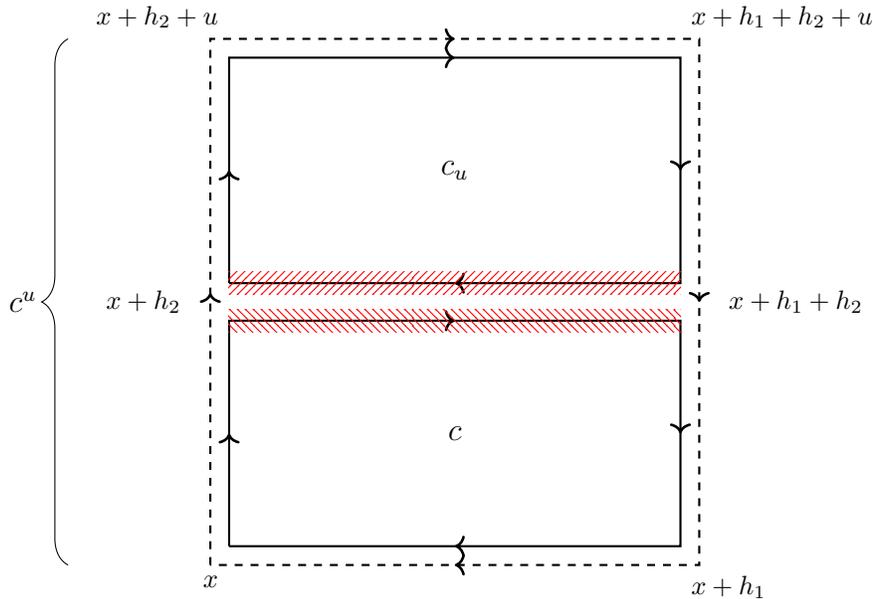
\begin{figure}[h]\label{figco}
\begin{tikzpicture}

\begin{scope}[decoration={markings,
mark=at position 1.5cm with {\arrow[line width=1pt]{>}},
mark=at position 6cm with {\arrow[line width=1pt]{>}},
mark=at position 10.5cm with {\arrow[line width=1pt]{>}},
mark=at position 15cm with {\arrow[line width=1pt]{>}}
}]
    
\path[draw,line width=0.8pt,postaction=decorate] (-2,0.5) -- (-2,3.5) -- (4,3.5) -- (4,0.5) -- (-2,0.5);
\end{scope}

\begin{scope}[decoration={markings,
mark=at position 1.5cm with {\arrow[line width=1pt]{>}},
mark=at position 6cm with {\arrow[line width=1pt]{>}},
mark=at position 10.5cm with {\arrow[line width=1pt]{>}},
mark=at position 15cm with {\arrow[line width=1pt]{>}}
}]
    
\path[draw,line width=0.8pt,postaction=decorate] (-2,-3) -- (-2,0) -- (4,0) -- (4,-3) -- (-2,-3);
\end{scope}

\begin{scope}[decoration={markings,
mark=at position 3.62cm with {\arrow[line width=1pt]{>}},
mark=at position 10.25cm with {\arrow[line width=1pt]{>}},
mark=at position 17cm with {\arrow[line width=1pt]{>}},
mark=at position 23.75cm with {\arrow[line width=1pt]{>}},
}]
    
\path[draw,line width=0.8pt,dashed,postaction=decorate] (-2.25,-3.25) -- (-2.25,3.75) -- (4.25,3.75) -- (4.25,-3.25) -- (-2.25,-3.25);
\end{scope}

\fill[pattern=north west lines, pattern color=red] (-2,-0.15) rectangle (4,0.15);

\fill[pattern=north east lines, pattern color=red] (-2,0.35) rectangle (4,0.65);

\node[below right] at (4,-3.25) {\small $x+h_1$};
\node[below left] at (-2,-3.25) {\small $x$};
\node[left] at (-2.5,0.25) {\small $x+h_2$};
\node[above left] at (-2,3.75) {\small $x+h_2+u$};
\node[above right] at (4,3.75) {\small $x+h_1+h_2+u$};
\node[right] at (4.5,0.25) {\small $x+h_1+h_2$};
\node at (1,2) {$c_u$};
\node at (1,-1.5) {$c$};

\draw [decorate,decoration={brace,amplitude=10pt},xshift=-4pt,yshift=0pt]
(-4,-3.25) -- (-4,3.75) node [black,midway,xshift=-0.6cm] 
{$c^u$};
\end{tikzpicture}
\caption{The relationship between the cubes $c$, $c^u$, and $c_u$. With the natural orientations we can `glue' $c$ and $c_u$ together to form $c^u$, so that $c=c^u-c_u$. The cocycle property is that $\partial$ respects this relationship and $\partial f(c)=\partial f(c^u)-\partial f(c_u)$ for any function $f$.}
\end{figure}

Generalising to arbitrary dimensions, if $c=(x,h_1,\ldots,h_k)$ and $1\leq i\leq k$ then we define the two related cubes 
\begin{align*}
c^u &= (x; h_1,h_2,\ldots,h_{i-1},h_i+u,h_{i+1},\ldots,h_k)\\
c_u &= (x+h_i; h_1,h_2,\ldots,h_{i-1},u,h_{i+1},\ldots,h_k).
\end{align*}
That is, both are $k$-dimensional cubes, where the only changes have been to the basepoint $x$ and the component in direction $i$. Both are of the form $(x',h_1,\ldots,h_i',\ldots,h_k)$ such that $(x+h_i)-(x'+h_i')=u$. Note that this definition is dependent on the choice of direction $i$, which we suppress for notational ease, and which will always be clear from context. 

As in the 2-dimensional case it is a trivial matter of expanding out the definitions to check that, for any function $f$, cube $c\in H^{k+1}$, and $u\in H$, the cocycle identity $\partial f(c)=\partial f(c^u)-\partial f(c_u)$ holds. We call functions on the space of cubes which satisfy the cocycle identity for every cube $c\in H^{k+1}$ and $u\in H$ cocycles (so that, for any $f$, the derivative $\partial f$ is always a cocycle). We will need a robust version of this concept, where the cocycle identity may only hold for 99\% of all possible cubes, giving the following definition.

\begin{definition}[Cocycles]
Let $\delta>0$ and $A$ be an abelian group. If $S\subset H^{k+1}$ is a set of cubes then a function $\rho:S\to A$ is a $k$-cocycle on $S$ with loss $\delta$ if, for every direction $1\leq i\leq k$, the cocycle identity
\[\rho(c) = \rho(c^u) - \rho(c_u)\]
is true for all but at most $\delta \lvert H\rvert^{k+2}$ pairs $(\bvec{c};u)\in H^{k+2}$ such that $c,c^u,c_u\in S$.
\end{definition}

In particular, the previous discussion shows that any derivative $\partial f$ (of order $k$) is a $k$-cocycle.  Remarkably, the converse is also essentially true -- one can view this `algebraic' inverse result as the seed from which our titular `combinatorial' inverse results grow.

\begin{lemma}\label{cocycleinv}
If $S_0,\ldots,S_{k}$ is a system of cubes which are 99\% dense and $\rho:S_k\to\mathbb{Z}^d$ is a $k$-cocycle with 1\% loss on $S_k$ then there exists $\lambda:H\to\mathbb{R}^d$ with $\partial \lambda: H^{k+1}\to \bbz^d$ such that
\begin{enumerate}
\item $\partial\lambda$ agrees with $\rho$ for 99\% of cubes in $S_k$ and
\item $\sup_{x\in H}\| \lambda(x)\|_1 \ll_k \sup_{c\in S_k}\| \rho(c)\|_1 +d$.
\end{enumerate}
\end{lemma}
It is important to note that, even where $\partial \lambda$ does not agree with $\rho$, the lemma guarantees that it is still an integer-valued function, and hence $\lambda$ behaves like a polynomial of degree $k$ `modulo $\bbz$'.
\begin{proof}[Sketch proof]
The key observation is that if we let 
\[\lambda(x) = \bbe_{\bvec{h}\in H^k} \rho(x;\bvec{h}),\]
then a lengthy but routine calculation using the definitions proves that $\partial \lambda=\rho$, at least when $\rho$ is a 100\% cocycle.  Some highly non-trivial combinatorial smoothing allows one to show that this identity still holds 99\% of the time when $\rho$ is a 99\% cocycle.
\end{proof}

\subsection{Generalised cocycles}
To motivate the introduction of generalised cocycles, we first sketch how the algebraic inverse result of the previous section may be applied to obtain a combinatorial inverse result when $s=1$.

Recall that our starting point is some function $f$ with 99\% structure, quantified by being at the top of a $(D,M)$-polynomial hierarchy of height $s=1$, say with coefficient function $b:H^3\times\{0,1\}^2\to \bbz^D$. Our goal is to show that the coefficient function $b(c,\omega)$ behaves like the carry bits from a bracket linear function, which is enough for an inverse result as mentioned in the introduction to this section. 

Unpacking the definitions (and assuming all 99\% objects are in fact 100\% objects for simplicity), we have a constant function $f_0:H\to \mathbb{R}^{D}$ such that for all $c\in H^{3}$, 
\begin{align*}
\partial f(c) &= \sum_{\omega_1,\omega_2\in \{0,1\}}(-1)^{\omega_0+\omega_1} f(x+\omega_1h_1+\omega_2h_2)\\
&= (b(c,00)-b(c,01)-b(c,10)+b(c,11))\cdot f_0.
\end{align*}
Since $f_0$ is constant, we can consolidate the first factor into a single function $B:H^3\to \mathbb{Z}^{D}$, expressing the right-hand side as $B(c)\cdot f_0$. In particular, when $s=1$, we can assume that the coordinate function does not depend on $\omega$ (note that the same consolidation takes place for the derivative of a bracket linear function). From the definition alone, the only knowledge we have about the function $B$ is that is takes on values which are $L^1$-bounded (by the parameter $M$). 

Since $B(c)\cdot f_0$ is a derivative, the cocycle identity must hold. If we can `invert' the $\cdot f_0$ then $B$ must itself satisfy the cocycle identity. By the algebraic inverse result Lemma~\ref{cocycleinv}, $B$ must itself, essentially, be a derivative, and so there is some $f':H\to \mathbb{R}^{D}$ such that $\partial f':H^3\to \bbz^D$ (so that $f'$ is like a linear polynomial modulo $\bbz$) and $\partial f = (\partial f')\cdot f_0$. This is still rather weak algebraic structure, but it is enough to be bootstrapped into constructing a genuine nilpolynomial which agrees with $f$ 1\% of the time. At this point the arguments become more algebraic than combinatorial, and as mentioned above, will not be discussed any further in this article.

Instead, we will take that final algebraic step for granted, and view obtaining information like `the coefficients $b$ of the polynomial hierarchy are themselves derivatives of some other function, which has some polynomial-like structure' as our goal. We will now indicate how the previous discussion generalises to higher $s\geq 2$. The key input when $s=1$ was that a function closely related to the coefficients $b$ satisfied the cocycle condition, since it is a derivative, and so we could apply the algebraic inverse result Lemma~\ref{cocycleinv}. For the more general argument, we must find some property of the coefficients $b$ that follows from the fact that they are related to a derivative, which satisfies the cocycle condition. This property Manners calls a `generalised cocycle condition', and can be derived heuristically as follows.

Suppose that $f$ sits on top of a polynomial hierarchy of height $s$, so that  for any $c= (x,\bvec{h})\in S_{s+1}$, 
\[\partial f(c)=\sum_{\omega\in\{0,1\}^{s+1}}(-1)^{\abs{\omega}}b(c,\omega) \cdot \bvec{f}_{\leq s-1}(x+\omega\cdot \bvec{h}).\]
We now see what this identity combined with the cocycle property of derivatives gives us. Let $(c,u)\in H^{s+3}$ be such that $c^u,c_u\in S_{s+1}$ (we fix an arbitrary direction $1\leq i\leq s+1$). The cocycle identity implies that
\[\partial f(c)=\sum_{\omega\in\{0,1\}^{s+1}}(-1)^{\abs{\omega}}\brac{b(c^u,\omega)\cdot\bvec{f}_{\leq s-1}(x+\omega\cdot \bvec{h}^u)-b(c_u,\omega)\cdot\bvec{f}_{\leq s-1}(x+h_i+\omega\cdot \bvec{h}_u)}.\]
If we let
\begin{equation}\label{bdef}
\mathfrak{b}(c,u,\omega)=\begin{cases} b(c^u,\omega)&\textrm{ if }\omega_i=0\textrm{ and}\\ b(c_u,\omega\backslash \{i\})&\textrm{ if }\omega_i=1\end{cases}
\end{equation}
then we can regroup this as
\begin{align*}
\partial f(c) &=\sum_{\omega\in\{0,1\}^{s+1}}(-1)^{\abs{\omega}}\mathfrak{b}(c,u,\omega)\cdot \bvec{f}_{\leq s-1}(x+\omega\cdot \bvec{h}) \\
&\hspace{6em} + \sum_{\substack{\omega\in \{0,1\}^{s+1}\\ i\in \omega}}(-1)^{\abs{\omega}}(b(c^u,\omega)-b(c_u,\omega))\cdot\bvec{f}_{\leq s-1}(x+\omega\cdot \bvec{h}).
\end{align*}
If $b(c^u,\omega)=b(c_u,\omega)$, then the second sum vanishes and
\[\sum_{\omega\in\{0,1\}^{s+1}}(-1)^{\abs{\omega}}b(c,\omega)\cdot \bvec{f}_{\leq s-1}(x+\omega\cdot \bvec{h})=\sum_{\omega\in\{0,1\}^{s+1}}(-1)^{\abs{\omega}} \mathfrak{b}(c,u,\omega)\cdot\bvec{f}_{\leq s-1}(x+\omega\cdot \bvec{h}) .\]
Comparing coefficients we may then hope that $b(c,\omega)=\mathfrak{b}(c,u,\omega)$. This, roughly, is the idea of the `generalised cocycle' condition. 

There are two technical matters to clarify before we give the precise definition of a generalised cocycle. The first is that in the final step we illegitimately concluded from `comparing coefficients' that $b(c,\omega)=\mathfrak{b}(c,u,\omega)$ for fixed $c,u,\omega$. The second is that we assumed that $b(c^u,\omega)=b(c_u,\omega)$. Neither of these are, in general, true, but can be recovered after some work if for each fixed cube $c$ the functions $b(c,\cdot)$ enjoy a special kind of structure, namely being in `normal form'. 
\begin{definition}[Normal form]
Let $t\leq k$. A function $b:\{0,1\}^k \to \bbz^{d_0+\cdots+d_t}$ is in normal form if whenever $\abs{\omega}=\omega_1+\cdots+\omega_k=j$ the first $d_{0}+\cdots+d_{j-1}$ coordinates of $b(\omega)$ are all zero. 
\end{definition}
Although we can certainly not assume that the coefficient function $b$ is in normal form itself, it can be transformed into one, incurring only a small cost. It is an easy exercise to check that, for any $b:\{0,1\}^k\to \bbz^d$, the function
\begin{equation}\label{ztrans}
b'(\omega')=\sum_{\omega\in\{0,1\}^k}Z_r(\omega,\omega')b(\omega)
\end{equation}
vanishes whenever $\abs{\omega'}>r$, where\footnote{We identify vectors in $\{0,1\}^k$ with subsets of $\{1,\ldots,k\}$ when we say $\omega\supset \eta\supset\omega'$.}
\[Z_r(\omega,\omega')=\sum_{\substack{\eta\in \{0,1\}^k\\ \omega\supset \eta \supset \omega'\\ \abs{\eta}\leq r}} (-1)^{\abs{\omega}-\abs{\eta}}.\]
By applying this transformation to each level of the coefficient function in a polynomial hierarchy, and keeping track of the constants using the fact that $\abs{Z_r}=O_s(1)$, we obtain the following lemma, which lets us assume that the coefficient function is in normal form.

\begin{lemma}
If $g:X\to\bbr$ sits at the top of a polynomial hierarchy  of height $s$ with an associated coefficient function $b:S_{s+1}\times\{0,1\}^{s+1}\to \bbz^{d_0+\cdots+d_{s-1}}$ then we can assume that the function $b(c,\cdot)$ is in normal form for every $c\in S$, at the cost of replacing the $M$ parameter by some $O_s(M)$.
\end{lemma}

We can therefore assume that $b$ is in normal form. We would still need to know that $\mathfrak{b}$ is in normal form also, but we can get around this by applying the transform \eqref{ztrans} to $\mathfrak{b}$, putting it into something like normal form. 

Before we define a generalised cocycle we need to quantify the condition that $b(c^u,\omega)=b(c_u,\omega)$ for `most' $c$ and $u$ that was used in our heuristic discussion. 
\begin{definition}
If $S\subset H^{k+1}$ is a set of $k$-dimensional cubes then $b:S\times\{0,1\}^k\to\bbz^d$ is $\delta$-almost upper compatible on $S$ if, for every direction $1\leq i\leq k$ and $\omega\in \{0,1\}^k$ such that $\omega_i=1$, the identity
\[b(c^u,\omega)=b(c_u,\omega)\]
holds for all but at most $\delta\abs{H}^{k+2}$ pairs $(c,u)\in H^{k+2}$ such that $c^u,c_u\in S$. 
\end{definition}
As mentioned above, generically we expect the coefficients $b$ produced in the proof to be upper compatible, so we will not dwell on this issue here. The case when they are not upper compatible cannot be ruled out, and adds yet more technicalities, but it is somewhat pathological, and can be dealt with by other means.

We now come to the crucial definition, that of a generalised cocycle. This definition is informed by our heuristic discussion above, the idea being that it captures what the cocycle identity implies for the coefficient functions of a polynomial hierarchy.

\begin{definition}[Generalised cocycle]
Let $b:S\times \{0,1\}^k\to\bbz^d$ be a function in normal form. If $0\leq r<k$ and $S\subset H^{k+1}$ is a set of cubes then we say that $b$ is a generalised $k$-cocycle of type $r$ and loss $\delta>0$ on $S$ if
\begin{enumerate}
\item $b$ is $\delta$-almost upper compatible on $S$ and
\item in every direction $1\leq i\leq k$, for all but at most $\delta \abs{H}^{k+2}$ many $(c,u)\in H^{k+2}$ such that $c^u,c_u\in S$, for all $\omega\in \{0,1\}^k$, 
\[b(c,\omega)= \tilde{\mathfrak{b}}_r(c,u,\omega')\]
where $\mathfrak{b}$ is given by \eqref{bdef} and the transform $\tilde{\mathfrak{b}}_r$ was defined in \eqref{ztrans}.
\end{enumerate}
\end{definition}

There is a strong structural result available for generalised cocycles, just as there is for cocycles. Indeed, one can check that if $\lambda$ is a `polynomial modulo $\bbz$' of degree $k-r-1$ (in that $\partial \lambda(c)$ vanishes modulo $\bbz$ for all cubes of dimension $k-r$) then, for any $c=(x,\bvec{h})\in H^{k+1}$, the function $b(\omega)=\lambda(x+\omega\cdot \bvec{h})$ is a generalised $k$-cocycle of type $r$ (at least, after passing to a normal form of $b$ as in \eqref{ztrans}). The following lemma says that the converse holds. 

\begin{lemma}
Let $S_0,\ldots,S_k$ be a system of cubes with 99\% density and $b:S_{k}\times \{0,1\}^k\to \bbz^d$ be a generalised $k$-cocycle of type $r$ and 1\% loss. 

There is a function $\lambda:H\to \bbr^d$ such that, if $\Lambda(c,\omega)$ is the transformation (as defined in \eqref{ztrans}) of the function $(c,\omega)\mapsto \lambda(x+\omega\cdot \bvec{h})$, then
\begin{enumerate}
\item $b(c,\omega)=\Lambda(c,\omega)$ for all $\omega\in\{0,1\}^k$ for 99\% of cubes  $c\in S_k$,
\item $\partial \lambda (c)\in \bbz^d$ for all $c\in H^{k-r+1}$, and
\item $\sup_{x\in H}\norm{\lambda(x)}_1 \ll \sup_{c\in S_k} \norm{b(c)}_1+d$.
\end{enumerate}
\end{lemma}

\subsection{Concluding the proof}
The following is a convenient package definition, which couples the derivatives condition of the previous section with the extra restriction that the coefficients are generalised cocycles. 

\begin{definition}[Strong derivatives condition]
Let $\bvec{f}=(f_i: X\to \bbr^{d_i})_{0\leq i\leq s}$ be a tuple of functions, let $0\leq k,t\leq s+1$, and $b:S\times \{0,1\}^k\to \bbz^{d_0+\cdots+d_{t-1}}$ where $S\subset H^{k+1}$ is a set of cubes.

A function $g:X\to\bbr$ satisfies the $(\bvec{f},k,t,M,\delta)$-strong derivatives condition on $S$ with coefficients $b$ if
\begin{enumerate}
\item $g$ satisfies the $(\bvec{f},k,t,M)$-derivatives condition on $S$ with coefficients $b$;
\item $b$ is in normal form;
\item if we write $b=(b_0,\ldots,b_{t-1})$ where $b_i:S\times \{0,1\}^k\to \bbz^{d_i}$ then $b_{i}$ is $\delta$-almost upper compatible on $S$ for $1\leq i\leq t-1$; and 
\item $b_{t-1}$ is a generalised $k$-cocycle of type $t-1$ and loss $\delta$ on $S$. 
\end{enumerate}
\end{definition}

From our heuristic discussion (that is, combining the cocycle property of derivatives with the definition of the derivatives condition) we expect that the coefficients within a polynomial hierarchy should have coefficients that are generalised cocycles, and in particular it should satisfy the strong derivatives condition with respect to this hierarchy. 

The following lemma is a precise statement of this fact. There are many aspects we have shrugged off or simply ignored in our heuristic discussion, such as the fact that $b$ may not be even remotely upper compatible, or the fact that all our statements hold only 99\% of the time rather than 100\% of the time. Perhaps the most significant, however, is the hierarchical nature of our structure. Unless one proceeds extremely carefully, the 99\% amount of structure enjoyed at the top level may only translate to a 1\% amount of structure on the next level down, at which point all of our heuristics cease to be even approximately true. 

To avoid this, one must take care to expand this 1\% to 99\% before we try to locate generalised cocycles. This `expansion' is probably the most lengthy and difficult part of Manners' proof. Indeed, it is also quantitatively the most costly, and is the reason there are different bounds between the cases $s\leq 3$ and $s\geq 4$ in Theorem~\ref{approx} and thence Theorem~\ref{thmain}. 
\begin{lemma}\label{lemmaend}
Let $\epsilon>0$ and let $(S_0,\ldots,S_{s+1})$ be a system of cubes inside $H$ of density $1-\epsilon$. Suppose $g:S_0\to\bbr$ sits at the top of a $(D,M)$-polynomial hierarchy of height $s$ on $S_0,\ldots,S_{s+1}$.

There exists $M'\geq 1$ and a system of cubes $(S_0',\ldots,S_{s+1}')$ of density $1-2\epsilon$ such that $S_k'\subset S_k$ and $\abs{S_k\backslash S_k'}\leq \epsilon\abs{S_k}$ for each $k$ and either
\begin{enumerate}
\item $H$ has a proper subgroup of size at least $\abs{H}/M'$, or
\item $g$ sits at the top of a $(D,M')$-polynomial hierarchy $\bvec{f'}$ of height $s$ on $S_0',\ldots,S_{s+1}'$ which is an $M'$-bounded linear combination of $\bvec{f}$, and further $g$ satisfies the $(\bvec{f},s+1,s,M',\epsilon)$-strong derivatives condition on $S_{s+1}'$,
\end{enumerate}
where 
\[M'\leq \begin{cases} O(M/\epsilon)^{O(D)}&\textrm{ if }s\leq 2\textrm{ and}\\
O(M/\delta)^{O(D)^{O(D)}}&\textrm{ if }s\geq 3.\end{cases}\]
\end{lemma}

Note that if $H$ has no non-trivial subgroups (as is the case for Theorem~\ref{thmain}, when $H=\bbz_N$ for $N$ prime), then the first case cannot hold (for bounded $M'$), and hence we must be in the second case. If $H$ does have large subgroups, then we cannot hope to iteratively apply Lemma~\ref{lemmaend}, and hence the strategy of Manners breaks down. This is a significant restriction on the work of Manners, and in particular means that this approach cannot be used for inverse theorems over groups with many subgroups, such as $\mathbb{F}_p^n$ (for which see the recent alternative approach of \textcite{GM17,GM20}). 

We will attempt only a very brief sketch of the ideas involved. The reader is cautioned that many highly non-trivial difficulties lurk in the details. We may heuristically interpret the hypothesis as saying that $\partial g$ vanishes (up to lower order terms) on 99\% of all $(s+1)$-dimensional cubes. We then would like to argue as in our heuristic discussion above to deduce that the coefficients satisfy a generalised cocycle condition, and hence $g$ satisfies a strong derivative condition as required. 

Roughly speaking, this can be done \emph{unless} $\partial g$ vanishes on 1\% of $s$-dimensional cubes, at which point too much of the potential information gained about the coefficients is annihilated. The idea is to then show that we can upgrade this 1\% to a 99\%, so that in fact $\partial g$ vanishes on 99\% of all $s$-dimensional cubes, and we can repeat the argument from the beginning, but with $s$ reduced by $1$, and so on. We must win somewhere, since the functions at the bottom level of the hierarchy are simply constants, and therefore at some level we can exit with a strong derivatives condition as required. 

We conclude by sketching how we grow this 1\% to 99\%. The goal is to show that, assuming $\partial g=0$ on 99\% of $(s+1)$-dimensional cubes and $\partial g=0$ on 1\% of all $s$-dimensional cubes, either
\begin{enumerate}
\item $H$ has a large subgroup, or
\item $\partial g$ vanishes on 99\% of $s$-dimensional cubes.
\end{enumerate}
We assume that the first condition fails, and will attempt to transform the 1\% of cubes on which $\partial g$ vanishes to 99\% of cubes. There are two operations on cubes that are used to do so: gluing (where from $c^u$ and $c_u$ we produce $c$) and translation (where from $(x,\bvec{h})$ we produce $(x+u,\bvec{h})$). These are based on the trivial observations that:
\begin{enumerate}
\item the cocycle identity $\partial g(c)=\partial g(c^u)-\partial g(c_u)$ implies that if $\partial g$ vanishes on both $c^u$ and $c_u$ then it must also vanish on $c$ and
\item the fact that by definition
\[ \partial g(x,h_1,\ldots,h_k,u) = \partial g(x,h_1,\ldots,h_{k}) - \partial g(x+u,h_1,\ldots,h_{k})\]
implies that if $\partial g$ vanishes on both $(x,\bvec{h})$ and $(x,\bvec{h},u)$ then it also vanishes on $(x+u,\bvec{h})$.
\end{enumerate}
Since we are working inside a system of cubes with good extension properties, the second in particular implies that if $\partial g$ vanishes on $(x,\bvec{h})$ then there are many $u$ such that it vanishes for $(x+u,\bvec{h})$. We call this the translation of $(x,\bvec{h})$ by $u$.

These two observations, coupled with reasonable compatibility conditions on the system of cubes, combine to imply that the set of $s$-dimensional cubes for which $\partial g$ vanishes is closed under both gluing and translation. Methods from additive combinatorics can then be used to show that this closure property of the set of cubes in $H^{s+1}$ forces them to actually be 99\% of $H_0^{s+1}$ where $H_0\leq H$ is some subgroup. Since we are assuming that the original set of cubes is at least 1\% of all cubes, and also that there are no non-trivial large subgroups, it follows that $H_0=H$, and hence in fact $\partial g$ must vanish on 99\% of all cubes. 

We remark that this final part of the argument is, in some sense, quantitatively the most inefficient, and is the reason for the double exponential in Theorem~\ref{thmain} when $s\geq 4$. Manners suspects that a more refined argument could remove this loss, which would mean that Theorem~\ref{thmain} holds for all $s\geq 1$ with only singly exponential bounds. 

\section{Summary of the proof}\label{secsummary}

We conclude this article by giving a summary of the entire proof of Manners to reorient the reader. We are given some function $f:\bbz_N\to \bbc$ such that $\norm{f}_{U^{s+1}}\geq \delta$.
\begin{enumerate}
\item There is a set $X\subset \bbz_N^{s-2}$, which is 1\% of all tuples, such that if $\bvec{h}\in X$ then $\Delta_{\bvec{h}}f$ has large $U^2$ norm;
\item The $U^2$ inverse theorem (a simple consequence of orthogonality) implies that there is some function $\Psi: X\to \bbr$ such that each $\Delta_{\bvec{h}}f$ correlates with $e(\Psi(\bvec{h})x)$;
\item The function $\Psi$ is an approximate polynomial of degree $s$, in that $\partial \Psi$ vanishes for 1\% of cubes $c\in (\bbz_N^{s-2})^{s+2}$;
\item There is another function $\Psi'$, which agrees with $\Psi$ 1\% of the time, such that 99\% of the time the $(s+1)$-derivatives of $\Psi'$ vanish, up to lower order terms (i.e. $\Psi'$ sits at the top of a polynomial hierarchy);
\item The lower order terms in the polynomial hierarchy are linear combinations of functions with coefficients that satisfy the generalised cocycle property;
\item Functions which satisfy the generalised cocycle property arise from the derivatives of functions which are polynomials modulo $\bbz$;
\item Functions which are at the top of a polynomial hierarchy, whose coefficients satisfy a generalised cocycle property, must agree some of the time with a nilpolynomial (that can be explicitly constructed from the coefficients);
\item Thus there is a nilpolynomial $P:\bbz_N^{s-2}\to \bbr$ such that, for 1\% of tuples $\bvec{h}\in \bbz_N^{s-2}$, the multiplicative derivative $\Delta_{\bvec{h}}f$ correlates with $e(P(\bvec{h})x)$. From this via an `integration and projection' argument, an explicit nilsequence $\Phi$ of degree $s$ can be constructed, such that $f$ correlates with $\Phi$.
\end{enumerate}
\printbibliography

\end{document}
